\documentclass[12pt,reqno]{amsart}
\usepackage[left=3cm,top=2cm,right=3cm,bottom=2cm]{geometry}
\usepackage{amssymb}
\usepackage{amsbsy}
\usepackage{epsfig}
\usepackage{tikz}
\usepackage[english]{babel}    
\usepackage{enumitem}
\usepackage{multicol}
\usepackage{blkarray}
\usepackage{tikz}
\usepackage{mathtools}
\usepackage{amsthm}
\usepackage{float}
\usepackage{algorithm}
\usepackage{algpseudocode}

\usepackage{graphicx}

\newtheorem{teore}{Theorem}[section]

\newcommand{\subia}{{\bf subcase~1a}}

\newcommand{\subiia}{{\bf subcase~2a}}

\newcommand{\myvar}{x}

\newtheorem{defn}[teore]{Definition}
\newtheorem{lemat}[teore]{Lemma}
\newtheorem{coro}[teore]{Corollary}
\newtheorem{conj}[teore]{Conjecture}

\newtheorem{remark}[teore]{Remark}
\newtheorem{Example}[teore]{Example}

\begin{document}
\DeclarePairedDelimiter\ceil{\lceil}{\rceil}
\DeclarePairedDelimiter\floor{\lfloor}{\rfloor}

\title[Index]{On $I$-eigenvalue free threshold graphs}

\author[Allem, Oliveira and Tura]{Luiz Emilio Allem{\footnotesize$^\dagger$}, Elismar R. Oliveira{\footnotesize$^\dagger$} \and Fernando Tura{\footnotesize$^\ddagger$}}

  %

\maketitle

\vspace*{-8mm}
\begin{center}

\footnotesize{ $^\dagger$ Instituto de Matem\'{a}tica e Estat\'{i}stica, UFRGS, 91509–900 Porto Alegre, RS, Brazil}
\vspace{2mm}

\footnotesize{ $^\ddagger$ Departamento de Matemática, UFSM, 97105–900 Santa Maria, RS, Brazil}
\vspace{2mm}

{\tt emilio.allem@ufrgs.br},~~~ {\tt elismar.oliveira@ufrgs.br}, ~~~{\tt ftura@smail.ufsm.br}
\end{center}
\vspace{6mm}


\begin{abstract}
A graph is said to be $I$-eigenvalue free if it has no eigenvalues in the interval $I$ with respect to the adjacency matrix $A$. In this paper we present two algorithms for generating $I$-eigenvalue free threshold graphs.

\end{abstract}


\section{Introduction}

Since the remarkable study of A. J. Hoffman \cite{hoffman1972limit} about the density of eigenvalues on the real line, several interesting results have been obtained on this topic as you can see in \cite{estes1992eigenvalues,salez2015every,zhang2006limit}. From this subject, a new area of research has arosen, finding classes of graphs that are $I$-eigenvalue free, where this paper stands.

A threshold graph can be constructed through an iterative process which starts with an isolated vertex, and at each step either a new isolated vertex is added or a dominating vertex is added, i.e., a vertex adjacent to all previous vertices is added. This process can be represented by a binary sequence that will be detailed in Section \ref{rep}.

This class of graphs was introduced by Chv\'{a}tal and Hammer \cite{chvtal1977aggregation} and Henderson and Zalestein \cite{henderson1977graph} in 1977 and their numerous applications go from computer science to psychology \cite{mahadev1995threshold}.

This paper deals with spectral properties of threshold graphs related to the adjacency matrix.  In \cite{sciriha2011spectrum} Sciriha and Farrugia proved that all eigenvalues, other than $-1$ and $0$, are main. Jacobs et al. \cite{jacobs2015eigenvalues} showed that threshold graphs are $(-1,0)$-eigenvalue free. A subclass of threshold graphs is the family of anti-regular graphs which are the graphs with only two vertices of equal degrees. In \cite{aguilar2018spectral} Aguilar et al. proved that anti-regular graphs have no eigenvalues in the interval $\Omega=\left[ \frac{-1-\sqrt{2}}{2},\frac{-1+\sqrt{2}}{2}\right]$ except the trivial eigenvalues $-1$ and $0$. This result was extended by Ghorbani in \cite{ghorbani2019eigenvalue} by showing that threshold graphs have no eigenvalues in $\Omega$ except the trivial ones.

Searching for spectral properties of graphs, in this paper, we provide two algorithms for generating infinite families of $I$-eigenvalue free threshold graphs. More specifically, given a scalar $N$ and a natural number $r$ we construct a threshold graph $G$ with associated cotree $T_{G}(a_{1},a_{2},\ldots,a_{r})$ of depth $r$, which will be the initial threshold graph for generating infinite families of $(0,N]$-eigenvalue free threshold graphs. The same result is obtained for the interval $[M,-1)$. We show that our approach also can be used to prove that threshold graphs are $\left(\frac{-1-\sqrt{2}}{2},-1\right]$ and $\left[0,\frac{-1+\sqrt{2}}{2}\right)$-eigenvalue free which shows that we have a generalization of the above mentioned results.

Here is an outline of the paper.
In Section \ref{rep} we describe the cotrees associated to threshold graphs, and in Section \ref{AlgDiag} we present some important results from \cite{jacobs2018eigenvalue}. In Sections \ref{Inertia0} and \ref{Inertia-1} we explain our strategy based on the inertia of the graph. The algorithms for generating infinite families of $I$-eigenvalue free threshold graphs are given in Section \ref{secN} and \ref{secM}. In Section \ref{Family} we exhibit such families. Our method is used to prove that threshold graphs are $\left(\frac{-1-\sqrt{2}}{2},-1\right]$ and $\left[0,\frac{-1+\sqrt{2}}{2}\right)$-eigenvalue free in Section \ref{secrec}. Finally, in Section \ref{Future} we discuss about some open problems.

\section{Threshold representation}
\label{rep}
Let $G=(V,E)$ be an undirected graph with vertex set $V$ and edge set $E$. If $|V|=n$, then its adjacency matrix $A(G)=[a_{ij}]$ is the $n\times n$ matrix of zeros and ones such that $a_{ij}=1$ if and only if $v_{i}$ is adjacent to $v_{j}$ (that is, there is an edge between $v_{i}$ and $v_{j}$). For $v\in V$, $N(v)$ denotes the open neighborhood of $v$, that is $\{w|\{v,w\}\in E\}$ and $N[v]=N(v)\cup\{v\}$ the closed neighborhood. A value $\lambda$ is an eigenvalue if det$(\lambda I_{n}-A)=0$, and since $A$ is real symmetric its eigenvalues are real. In this paper, a graph's eigenvalues are the eigenvalues of its adjacency matrix.\\

We may represent a threshold graph on $n$ vertices using a binary sequence $(b_{1},b_{2},\ldots,b_{n})$ as follows: $b_{i}=0$ if vertex $v_{i}$ is added as an isolated vertex, and $b_{i}=1$ if $v_{i}$ is added as a dominating vertex. In Figure \ref{binary} we show the threshold graph $G$ with binary sequence $b=(1,1,1,1,0,0,0,1,1)$. It can also be seen as consecutive blocks of $0$'s and $1$'s, for instance $b=(1,1,1,1,0,0,0,1,1)=1^{4}0^{3}1^{2}$. Notice that we order the vertices of $G$ in the same way they are given in their sequence.

%
%

    %
    %
    %
    %


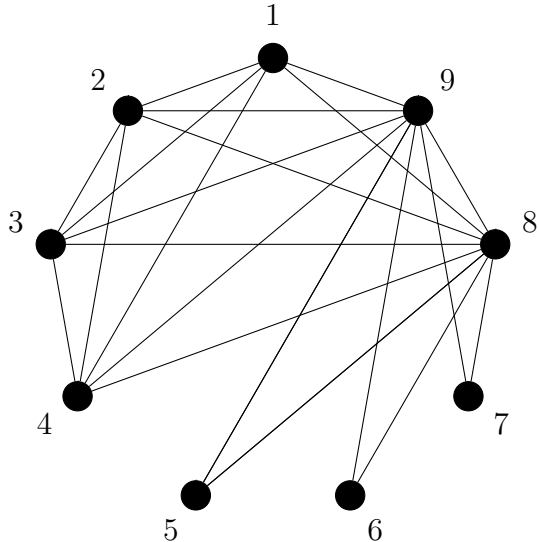
\begin{figure}[H]
\begin{tikzpicture}
  [scale=1,auto=left,every node/.style={circle}]
  \foreach \i/\w in {1/,2/,3/,4/,5/,6/,7/,8/,9/}{
    \node[draw,circle,fill=black,label={360/9 * (\i - 1)+90}:\i] (\i) at ({360/9 * (\i - 1)+90}:3) {\w};} 
  \foreach \from in {2,3,4,8,9}{
    \foreach \to in {1,2,...,\from}
      \draw (\from) -- (\to);}
       \foreach \from in {5}{
       \foreach \to in {8,9,\from}
       \draw (\from) -- (\to)
       ;}
\end{tikzpicture}
       \caption{ $b=(1,1,1,1,0,0,0,1,1)=1^4 0^3 1^2$}
       \label{binary}
 \end{figure}

Threshold graphs are a subclass of cographs and each cograph can be represented by a cotree, for more details see \cite{biyikoglu2011some,jacobs2018eigenvalue}. It is interesting that for threshold graphs, the cotree is a caterpillar, as shown in \cite{sciriha2011spectrum}.\\

In this note, we focus on representing the threshold graph using its cotree (caterpillar), that we describe below. A cotree $T_{G}$ of a threshold graph is a rooted path in which any interior vertex $w$ is either of union $\cup$ type (corresponding to a block of $0$'s) or join $\otimes$ type (corresponding to a block of $1$'s). The terminal vertices (leaves) are typeless and represent the vertices of the threshold. Since we work only with connected thresholds graphs in this paper, our cotree is basically defined by placing a $\otimes$ node at the trees's root. And then, placing $\cup$ on interior nodes with odd depth, and placing $\otimes$ on interior nodes with even depth.\\

Notice that, if a cotree $T_{G}$ associated to a threshold graph has an even depth then its final interior node is a $\cup$ type, and if its depth is odd then it is a $\otimes$ type as in Figure \ref{cotree}.

The cotree denoted by $T_{G}(a_{1},\, a_{2},\, \ldots,\, a_{r})$ with $a_{r}\geq 2$ and $r$ odd is depicted in Figure \ref{cotree}. Notice that following our notation each interior node at depth $i$ has $a_{i}$ terminal vertices (leaves).

\begin{figure}[H]
\centering
\begin{tikzpicture}

\path( 0.5,1)node[shape=circle,draw,label=above:$r$,inner sep=0] (pri) {$\otimes$}
     ( 0,0)node[shape=circle,draw,fill=black] (pri1) {}
     ( 1,0)node[shape=circle,draw,fill=black] (pri2) {}
     ( 2,1)node[shape=circle,draw,label=above:$r-1$,inner sep=0] (sec) {$\cup$}
     ( 1.5,0)node[shape=circle,draw,fill=black] (sec1) {}
     ( 2.5,0)node[shape=circle,draw,fill=black] (sec2) {}
     ( 4,1)node[shape=circle,draw,label=above:$2$,inner sep=0] (ter) {$\cup$}
     ( 3.5,0)node[shape=circle,draw,fill=black] (ter1) {}
     ( 4.5,0)node[shape=circle,draw,fill=black] (ter2) {}
     ( 5.5,1)node[shape=circle,draw,label=above:$1$,inner sep=0] (qua) {$\otimes$}
     ( 5,0)node[shape=circle,draw,fill=black] (qua1) {}
     ( 6,0)node[shape=circle,draw,fill=black] (qua2) {};

      \draw[-](pri)--(pri1);
      \draw[-](pri)--(pri2);
      \draw[dotted](pri1)--(pri2);
      \draw[-](sec)--(pri);
      \draw[-](sec)--(sec1);
      \draw[-](sec)--(sec2);
      \draw[dotted](sec1)--(sec2);

      \draw[-](ter)--(ter1);
      \draw[-](ter)--(ter2);
      \draw[dotted](ter1)--(ter2);

       \draw[dotted](ter)--(sec);

      \draw[-](qua)--(qua1);
      \draw[-](qua)--(qua2);
      \draw[dotted](qua1)--(qua2);

       \draw[-](ter)--(qua);

\end{tikzpicture}
\caption{\label{cotree} Cotree $T_{G}(a_{1},a_{2},\ldots,a_{r})$.}
\end{figure}
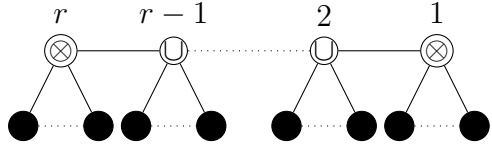

The binary representation corresponding to the cotree in Figure \ref{cotree} is given by $b=\left( 1^{a_{r}},b_{2}^{a_{r-1}},\ldots,b_{r-1}^{a_{2}},1^{a_{1}}\right)$. For instance, the threshold with binary representation  $b=(1,1,1,1,0,0,0,1,1)=$ $1^{4}0^{3}1^{2}$ has the cotree representation $T_{G}(2,3,4)$ depicted in Figure \ref{caterpillar}.

\begin{figure}[H]
\centering
\begin{tikzpicture}
\path(0.7,1)node[shape=circle,draw,label=above:$3$,inner sep=0] (pri) {$\otimes$}
     (0,0)node[shape=circle,draw,fill=black] (pri1) {}
     (0.5,0)node[shape=circle,draw,fill=black] (pri2) {}
     (1,0)node[shape=circle,draw,fill=black] (pri3) {}
     (1.5,0)node[shape=circle,draw,fill=black] (pri4) {}

     ( 2.5,1)node[shape=circle,draw,label=above:$2$,inner sep=0] (sec) {$\cup$}
     ( 2,0)node[shape=circle,draw,fill=black] (sec1) {}
     ( 2.5,0)node[shape=circle,draw,fill=black] (sec2) {}
     ( 3,0)node[shape=circle,draw,fill=black] (sec3) {}

     ( 4,1)node[shape=circle,draw,label=above:$1$,inner sep=0] (ter) {$\otimes$}
     ( 3.75,0)node[shape=circle,draw,fill=black] (ter1) {}
     ( 4.25,0)node[shape=circle,draw,fill=black] (ter2) {};

      \draw[-](pri)--(sec);
      \draw[-](pri)--(pri1);
      \draw[-](pri)--(pri2);
       \draw[-](pri)--(pri3);
      \draw[-](pri)--(pri4);

      \draw[-](sec)--(sec3);
      \draw[-](sec)--(sec1);
      \draw[-](sec)--(sec2);

      \draw[-](ter)--(ter1);
      \draw[-](ter)--(ter2);
      \draw[-](ter)--(sec);

\end{tikzpicture}
\caption{\label{caterpillar} Cotree $T_{G}(2,3,4)$.}
\end{figure}
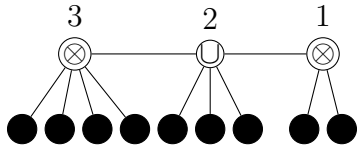

Finally, we define siblings vertices given its important role in the structure of cographs and in the following algorithm. Two vertices $u$ and $v$ are duplicates if $N(u)=N(v)$ and coduplicates if $N[u]=N[v]$. Therefore, we call $u$ and $v$ siblings if they are either duplicates or coduplicates.

\section{Algorithm Diagonalize}
\label{AlgDiag}

The algorithm presented in this section was developed in \cite{jacobs2018eigenvalue}. Basically, it constructs a diagonal matrix $D$ congruent to $A+xI$, where $A$ is the adjacency matrix of a cograph $G$.

\begin{algorithm}[H]
\caption{Diagonalize $(T_{G}, x)$}\label{alg0}
\begin{algorithmic}
\Require cotree $T_G$, scalar $\myvar$
\Ensure  diagonal matrix $D=[d_1, d_2, \ldots, d_n]$ congruent to $A(G) + \myvar I$
\State initialize $d_i := \myvar$, for $ 1 \leq i \leq n$

\While{$T_G$  has $\geq 2$    leaves }
   \State  select siblings $\{v_{k},v_{l}\}$ of maximum depth with parent $w$
   \State $\alpha \leftarrow  d_k$    $\beta \leftarrow d_{l}$
   \If{$ w=\otimes$}
      \If{$\alpha + \beta \neq2$} \verb+                //subcase 1a+
         \State $d_{l} \leftarrow \frac{\alpha \beta -1}{\alpha + \beta -2};$ \hspace*{0,25cm} $d_{k} \leftarrow \alpha + \beta -2; $\hspace{0,25cm}   $T_G = T_G - v_k$
       \ElsIf{ $\beta=1$} \verb+                //subcase 1b+
           \State $d_{l} \leftarrow 1$ \hspace*{0,25cm}   $d_k  \leftarrow 0;$ \hspace{0,25cm} $T_G = T_G - v_k$
         \Else \verb+                      //subcase 1c+
            \State $d_{l} \leftarrow 1$  \hspace*{0,25cm} $d_k \leftarrow -(1-\beta)^2;$ \hspace{0,25cm} $T_G= T_G -v_k;$ \hspace{0,25cm} $T_G = T_G -v_l$
       \EndIf

\ElsIf{$ w=\cup$}
      \If{$\alpha + \beta \neq 0$} \verb+                //subcase 2a+
         \State $d_{l} \leftarrow \frac{\alpha \beta }{\alpha + \beta};$ \hspace*{0,25cm} $d_{k} \leftarrow \alpha + \beta ; $\hspace{0,25cm}   $T_G = T_G - v_k$
       \ElsIf{ $\beta=0$} \verb+                //subcase 2b+
           \State $d_{l} \leftarrow 0$ \hspace*{0,25cm}   $d_k  \leftarrow 0;$ \hspace{0,25cm} $T_G = T_G - v_k$
         \Else \verb+                      //subcase 2c+
            \State $d_{l} \leftarrow \beta$  \hspace*{0,25cm} $d_k \leftarrow -\beta;$ \hspace{0,25cm} $T_G= T_G -v_k;$ \hspace{0,25cm} $T_G = T_G -v_l$
       \EndIf
\EndIf

\EndWhile

\end{algorithmic}
\end{algorithm}

We would like to point out that the Diagonalize$(T_{G},x)$ works bottom up since the cotree is represented in the same way. In this note we just work with threshold graphs so the cotrees we use are depicted in Figure \ref{cotree}. Therefore, throughout the text we represent the steps bottom up by steps from left to right.

In this article given a graph $G$ and a scalar $a\in\mathbb{R}$ we define the triple $(a_{+},a_{0},a_{-})$, where $a_{+}$ denotes the number of eigenvalues of $G$ that are greater than $a$, $a_{0}$ the multiplicity of $a$ and $a_{-}$ the number of eigenvalues of $G$ that are less than $a$. Therefore, the inertia of a graph $G$, using our notation, is the triple $(0_{+},0_{0},0_{-})$.
The following results presented in this section are from \cite{jacobs2018eigenvalue} and will be used throughout the note.
The next one shows that the Algorithm \ref{alg0} computes the triple $(a_{+},a_{0},a_{-})$.

\begin{teore}
\label{inertia}
  Let $D=[d_{1},d_{2},\ldots,d_{n}]$ be the diagonal returned by Diagonalize$(T_{G},-a)$, and assume $D$ has $a_{+}$ positive values, $a_{0}$ zeros and $a_{-}$ negative values. Then
  \begin{enumerate}
	\item[i:] The number of eigenvalues of $G$ that are greater than $a$ is exactly $a_{+}$.
    \item[ii:] The number of eigenvalues of $G$ that are less than $a$ is exactly $a_{-}$.
	\item[iii:] The multiplicity of $a$ is $a_{0}$ .
\end{enumerate}
\end{teore}

The following three results show that we can obtain information about the localization of certain eigenvalues of the cograph $G$ just by analysing its associated cotree $T_{G}$.

\begin{teore}
\label{teo4}
  Let $G$ be a cograph with cotree $T_{G}$ having $\otimes$-nodes $\{w_{1},\ldots,w_{j}\}$, and assume each $w_{i}$ has $t_{i}$ children in $T_{G}$. Then $$n_{-}(G)=\sum_{i=1}^{j}(t_{i}-1).$$
\end{teore}

\begin{teore}
\label{teo5}
  Let $G$ be a cograph with cotree $T_{G}$ having $\cup$-nodes $\{w_{1},\ldots,w_{m}\}$, where $w_{i}$ has $t_{i}>0$ terminal children. If $G$ has $j\geq 0$ isolated vertices, then
$$n_{0}(G)=j+\sum_{i=1}^{m}(t_{i}-1).$$
\end{teore}

\begin{teore}
\label{teo6}
  Let $G$ be a cograph with cotree $T_{G}$ having $\otimes$-nodes $\{w_{1},\ldots,w_{m}\}$, where $w_{i}$ has $t_{i}>0$ terminal children. Then the multiplicity of $-1$ is
     $$\sum_{i=1}^{m}(t_{i}-1).$$
\end{teore}

The next two lemmas provide an alternative initialization of the algorithm i.e.,  first we perform assignments to the leaves with identical value of the cotree and then we move on with the specialized cotree.

\begin{lemat}
\label{lem10}
If $v_{1},\ldots,v_{m}$ have parent $w=\otimes$, each with diagonal value $y\neq 1$, then the algorithm performs $m-1$ iterations of \subia, assigning during iteration $j$:
\begin{align}
    d_{k} & \leftarrow\frac{j+1}{j}(y-1) \label{crossperm}\\
    d_{l} & \leftarrow\frac{y+j}{j+1}.   \label{crossrem}
\end{align}
\end{lemat}

\begin{lemat}
\label{lem11}
If $v_{1},\ldots,v_{m}$ have parent $w=\cup$, each with diagonal value $y\neq 0$, then the algorithm performs $m-1$ iterations of \subiia, assigning during iteration $j$:
\begin{align}
d_{k} & \leftarrow\frac{j+1}{j}(y) \label{cupperm}\\
d_{l} &\leftarrow\frac{y}{j+1}.    \label{cuprem}
\end{align}

\end{lemat}

One advantage of the algorithm is that it is applied directly in the cotree so the diagonal values processed during the iterations will be labeled in the vertices. As pointed out before, the algorithm progresses bottom up but in this article we work only with threshold  graphs so we will represent these iterations from left to right. The next example elucidates these observations.

\begin{Example}
\label{example1}
We will apply Diagonalize to the cograph $G$ with cotree $T_{G}(2,3,4)$ and $x=-1$. Since we process the algorithm directly in the cotree, the diagonal values $d_{i}$'s will appear at the terminal vertices $v_{i}$'s during the execution. Each leaf is initialized with the value $d_{i}=-1$ as in Figure \ref{example}. Since multiple leaves of the same parent ($\otimes$ or $\cup$) have the same diagonal value $y=-1$ at the initialization, we can begin performing assignments to the terminal children (leaves) of the cotree using Lemmas \ref{lem10} and \ref{lem11} as follows.

\begin{figure}[H]
\centering
\begin{tikzpicture}
\path(0.7,1)node[shape=circle,draw,label=above:$3$,inner sep=0] (pri) {$\otimes$}
     (0,0)node[shape=circle,draw,fill=black,label=below:$-1$] (pri1) {}
     (0.5,0)node[shape=circle,draw,fill=black,label=below:$-1$] (pri2) {}
     (1,0)node[shape=circle,draw,fill=black,label=below:$-1$] (pri3) {}
     (1.5,0)node[shape=circle,draw,fill=black,label=below:$-1$] (pri4) {}

     ( 2.5,1)node[shape=circle,draw,label=above:$2$,inner sep=0] (sec) {$\cup$}
     ( 2,0)node[shape=circle,draw,fill=black,label=below:$-1$] (sec1) {}
     ( 2.5,0)node[shape=circle,draw,fill=black,label=below:$-1$] (sec2) {}
     ( 3,0)node[shape=circle,draw,fill=black,label=below:$-1$] (sec3) {}

     ( 4,1)node[shape=circle,draw,label=above:$1$,inner sep=0] (ter) {$\otimes$}
     ( 3.75,0)node[shape=circle,draw,fill=black,label=below:$-1$] (ter1) {}
     ( 4.25,0)node[shape=circle,draw,fill=black,label=below:$-1$] (ter2) {};

      \draw[-](pri)--(sec);
      \draw[-](pri)--(pri1);
      \draw[-](pri)--(pri2);
       \draw[-](pri)--(pri3);
      \draw[-](pri)--(pri4);

      \draw[-](sec)--(sec3);
      \draw[-](sec)--(sec1);
      \draw[-](sec)--(sec2);

      \draw[-](ter)--(ter1);
      \draw[-](ter)--(ter2);
      \draw[-](ter)--(sec);

\end{tikzpicture}
\caption{\label{example} Diagonalize$(T_{G}(2,3,4),-1)$.}
\end{figure}
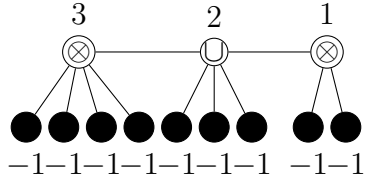

At depth $3$, there are $4$ leaves all with value $y=-1$, and using Lemma \ref{lem10}, $3$ iterations are performed. According to equation (\ref{crossperm}), $-4$, $-3$, $\frac{-8}{3}$ are the permanent diagonal values assigned. And, by (\ref{crossrem}), the final remaining value is  $$d_{l}\leftarrow\frac{1}{2}.$$
Figure \ref{example1n} depicts the cotree after Lemma \ref{lem10} being applied at $\otimes$-nodes at depth $1$ and $3$, and  Lemma \ref{lem11} at $\cup$-node at depth $2$.

\begin{figure}[H]
\centering
\begin{tikzpicture}
\path(1.5,1)node[shape=circle,draw,label=above:$3$,inner sep=0] (pri) {$\otimes$}
     (0,0)node[shape=circle,draw,fill=black,label=below:$-4$] (pri1) {}
     (1,0)node[shape=circle,draw,fill=black,label=below:$-3$] (pri2) {}
     (2,0)node[shape=circle,draw,fill=black,label=below:$\frac{-8}{3}$] (pri3) {}
     (3,0)node[shape=circle,draw,fill=red,label=below:$\frac{1}{2}$] (pri4) {}

     ( 6,1)node[shape=circle,draw,label=above:$2$,inner sep=0] (sec) {$\cup$}
     ( 5,0)node[shape=circle,draw,fill=black,label=below:$-2$] (sec1) {}
     ( 6,0)node[shape=circle,draw,fill=black,label=below:$\frac{-3}{2}$] (sec2) {}
     ( 7,0)node[shape=circle,draw,fill=red,label=below:$\frac{-1}{3}$] (sec3) {}

     ( 9.5,1)node[shape=circle,draw,label=above:$1$,inner sep=0] (ter) {$\otimes$}
     ( 9,0)node[shape=circle,draw,fill=black,label=below:$-4$] (ter1) {}
     ( 10,0)node[shape=circle,draw,fill=red,label=below:$0$] (ter2) {};

      \draw[-](pri)--(sec);

      \draw[-](pri)--(pri4);

      \draw[-](sec)--(sec3);

      \draw[-](ter)--(ter2);
      \draw[-](ter)--(sec);

\end{tikzpicture}
\caption{\label{example1n} Specialized Cotree $T_{G}(2,3,4)$.}
\end{figure}
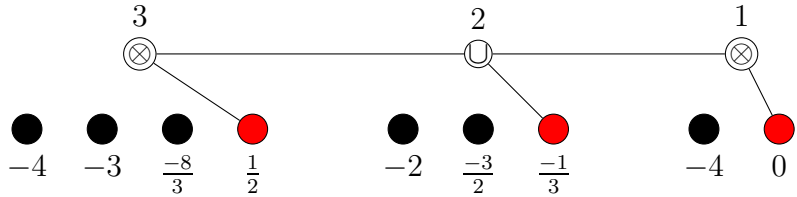

Notice that the vertices associated to permanent values are removed and we proceed with the cotree with the remaining vertices and its values as depicted in Figure
\ref{example2} left. The last vertex with value $\frac{1}{2}$ is relocated to the  next level as in Figure \ref{example2} right.

\begin{figure}[H]
\centering
\begin{tikzpicture}
\path(0,1)node[shape=circle,draw,label=above:$3$,inner sep=0] (pri) {$\otimes$}
     (0,0)node[shape=circle,draw,fill=red,label=below:$\frac{1}{2}$] (pri1) {}

     ( 1,1)node[shape=circle,draw,label=above:$2$,inner sep=0] (sec) {$\cup$}
     ( 1,0)node[shape=circle,draw,fill=red,label=below:$\frac{-1}{3}$] (sec1) {}

     ( 2,1)node[shape=circle,draw,label=above:$1$,inner sep=0] (ter) {$\otimes$}
     ( 2,0)node[shape=circle,draw,fill=red,label=below:$0$] (ter1) {};

      \draw[-](pri)--(sec);
      \draw[-](pri)--(pri1);

      \draw[-](sec)--(sec1);

      \draw[-](ter)--(ter1);

      \draw[-](ter)--(sec);

       \node[text width=6cm, anchor=west, right] at (4,0.5) {$\rightarrow$};

      \path(6.5,1)node[shape=circle,draw,label=above:$2$,inner sep=0] (sec) {$\cup$}
     ( 6,0)node[shape=circle,draw,fill=red,label=below:$\frac{1}{2}$] (sec1) {}
     ( 7,0)node[shape=circle,draw,fill=red,label=below:$\frac{-1}{3}$] (sec2) {}

     ( 8,1)node[shape=circle,draw,label=above:$1$,inner sep=0] (ter) {$\otimes$}
     ( 8,0)node[shape=circle,draw,fill=red,label=below:$0$] (ter1) {};

      \draw[-](sec)--(ter);
      \draw[-](sec)--(sec1);
      \draw[-](sec)--(sec2);

      \draw[-](ter)--(ter1);

\end{tikzpicture}
\caption{\label{example2} Applying \subiia.}
\end{figure}

Select the sibling pair $\{v_{k},v_{l}\}$ labeled with $d_{k}=\frac{1}{2}$ and $d_{l}=\frac{-1}{3}$ at the $\cup$-node at depth $2$ and initialize $\alpha\leftarrow d_{k}=\frac{1}{2}$ and $\beta\leftarrow d_{l}=\frac{-1}{3}$ as illustrated in Figure \ref{example2} right. Then \subiia~is executed and the assignments $d_{k}\leftarrow\frac{1}{6}$ and $d_{l}\leftarrow -1$ are made as shown in Figure \ref{example3} left. Then, the remaining vertex is relocated to the next level as in Figure \ref{example3} right.

\begin{figure}[H]
\centering
\begin{tikzpicture}
\path(1,1)node[shape=circle,draw,label=above:$1$,inner sep=0] (pri) {$\cup$}
     (1,0)node[shape=circle,draw,fill=red,label=below:$-1$] (pri1) {}
     (0,0)node[shape=circle,draw,fill=black,label=below:$\frac{1}{6}$] (pri2) {}

     ( 2,1)node[shape=circle,draw,label=above:$1$,inner sep=0] (sec) {$\otimes$}
     ( 2,0)node[shape=circle,draw,fill=red,label=below:$0$] (sec1) {};

      \draw[-](pri)--(sec);

      \draw[-](pri)--(pri1);

      \draw[-](sec)--(sec1);

      \node[text width=6cm, anchor=west, right] at (4,0.5) {$\rightarrow$};

      \path(6.5,1)node[shape=circle,draw,label=above:$2$,inner sep=0] (ter) {$\otimes$}
     ( 6,0)node[shape=circle,draw,fill=red,label=below:$-1$] (ter1) {}
     ( 7,0)node[shape=circle,draw,fill=red,label=below:$0$] (ter2) {};

     \draw[-](ter)--(ter1);
    \draw[-](ter)--(ter2);

\end{tikzpicture}
\caption{\label{example3} Applying \subia.}
\end{figure}

Now, as $\alpha=-1$ and $\beta=0$, \subia~is executed at the last iteration and the assignments $d_{k}\leftarrow -3$ and $d_{l}\leftarrow \frac{1}{3}$ are made as shown in Figure \ref{example4}.

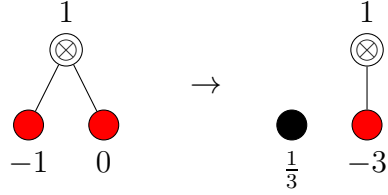
\begin{figure}[H]
\centering
\begin{tikzpicture}
\path(0.5,1)node[shape=circle,draw,label=above:$1$,inner sep=0] (pri) {$\otimes$}
     (0,0)node[shape=circle,draw,fill=red,label=below:$-1$] (pri1) {}
     (1,0)node[shape=circle,draw,fill=red,label=below:$0$] (pri2) {};

       \draw[-](pri)--(pri1);
        \draw[-](pri)--(pri2);

      \node[text width=6cm, anchor=west, right] at (2,0.5) {$\rightarrow$};

      \path(4.5,1)node[shape=circle,draw,label=above:$1$,inner sep=0] (ter) {$\otimes$}
     ( 4.5,0)node[shape=circle,draw,fill=red,label=below:$-3$] (ter1) {}
     ( 3.5,0)node[shape=circle,draw,fill=black,label=below:$\frac{1}{3}$] (ter2) {};

     \draw[-](ter)--(ter1);

\end{tikzpicture}
\caption{\label{example4} Final remaining value.}
\end{figure}

The algorithm stops, and the final diagonal is formed by the permanent values and the last remaining value as shown in Figure \ref{example5}. Therefore, according to Theorem \ref{inertia} $1_{+}=2$, $1_{-}=7$ and $1_{0}=0$.

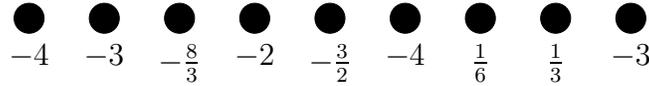
\begin{figure}[H]
\centering
\begin{tikzpicture}
\path(0,0)node[shape=circle,draw,fill=black,label=below:$-4$] (1) {}
     (1,0)node[shape=circle,draw,fill=black,label=below:$-3$] (2) {}
     (2,0)node[shape=circle,draw,fill=black,label=below:$-\frac{8}{3}$] (3) {}
     (3,0)node[shape=circle,draw,fill=black,label=below:$-2$] (4) {}
     (4,0)node[shape=circle,draw,fill=black,label=below:$-\frac{3}{2}$] (5) {}
     (5,0)node[shape=circle,draw,fill=black,label=below:$-4$] (6) {}
     (6,0)node[shape=circle,draw,fill=black,label=below:$\frac{1}{6}$] (7) {}
     (7,0)node[shape=circle,draw,fill=black,label=below:$\frac{1}{3}$] (8) {}
     (8,0)node[shape=circle,draw,fill=black,label=below:$-3$] (9) {};

\end{tikzpicture}
\caption{\label{example5} Final diagonal.}
\end{figure}

\end{Example}

\section{Interval $(0,N]$}
\label{Inertia0}

Given a real number $N>0$ our method consists in constructing a threshold graph $G$ with cotree $T_{G}(a_{1},\, a_{2},\, \ldots,\, a_{r})$, by choosing each $a_{i}\geq 1$ for $1\leq i\leq r-1$ and $a_{r}\geq 2$, and showing that this threshold graph $G$ is $(0,N]$-eigenvalue free using the fact  that $N_{+}=0_{+}$. Next we compute $0_{+}$ as a function  of the number of crosses and unions in $T_{G}$ for the cases $r$ odd and even .\\

Let $G$ be a Threshold graph with cotree $T_{G}(a_{1},\, a_{2},\, \ldots,\, a_{r})$ with $a_{r}\geq 2$ and $r$ odd. When we apply the algorithm Diag$(T_{G},0)$ we are able to compute the inertia $(0_{+},0_{0},0_{-})$ of $G$.
First, note that $r=|\otimes|+|\cup|=2|\otimes|-1$ since $|\cup|=|\otimes|-1$, where $|\otimes|$ and $|\cup|$ denotes, respectively, the number of crosses and unions in $G$ w.r.t. to the cotree representation. Hence, by Theorem \ref{teo4} we obtain the following
$$0_{-}=\left(\sum_{i=1}^{|\otimes|-1}a_{2i-1}+1-1\right)+a_{r}-1=\left(\sum_{i=1}^{|\otimes|}a_{2i-1}\right)-1.$$
And, by Theorem \ref{teo5}
$$0_{0}=\sum_{i=1}^{|\cup|}\left(a_{2i}-1\right)=\sum_{i=1}^{|\cup|}\left(a_{2i}\right)-|\cup|.$$
Finally
$$0_{+}=n-0_{-}-0_{0}=n-\left(\left(\sum_{i=1}^{|\otimes|}a_{2i-1}\right)-1\right)-\left(\sum_{i=1}^{|\cup|}\left(a_{2i}\right)-|\cup|\right)$$ $$0_{+}=n-\underbrace{\left(\left(\sum_{i=1}^{|\otimes|}a_{2i-1}\right)+\sum_{i=1}^{|\cup|}\left(a_{2i}\right)\right)}_{n}+|\cup|+1$$ $$0_{+}=|\cup|+1.$$

If $r$ is even, then $r=|\otimes|+|\cup|$ with $|\otimes|=|\cup|$. Therefore, as before, by Theorem \ref{teo4}
$$0_{-}=\sum_{i=1}^{|\otimes|}(a_{2i-1}+1-1)=\sum_{i=1}^{|\otimes|}(a_{2i-1})$$
and by Theorem \ref{teo5}
$$0_{0}=\sum_{i=1}^{|\cup|}(a_{2i}-1)=\sum_{i=1}^{|\cup|}(a_{2i})+|\cup|.$$
So,
$$0_{+}=n-0_{-}-0_{0}=n-\left(\sum_{i=1}^{|\otimes|}\left(a_{2i-1}\right)\right)-\left(\sum_{i=1}^{|\cup|}\left(a_{2i}\right)+|\cup|\right)=|\cup|.$$

Therefore, as aforementioned, our approach is to construct a threshold graph such that $N_{+}=|\cup|+1$ if $r$ is odd and $N_{+}=|\cup|$ if $r$ is even.

\section{Diag$(T_{G},-N)$}
\label{secN}

Instead of directly applying the algorithm Diag$(T_{G},-N)$ for $N>0$, we initialize the process using Lemmas \ref{lem10} and \ref{lem11} in the leaves of the cotree with identical value as described in Example \ref{example1}.

First, Lemma \ref{lem10} is applied at each $\otimes$-node at depth $i$ of the cotree. We perform $a_{i}-1$ iterations that leave negative permanent assignments $\frac{j+1}{j}(-N-1)$ for $j=1\ldots a_{i}-1$ by (\ref{crossperm}) and by (\ref{crossrem}) a remaining value
$$d_{l} \leftarrow\frac{-N+a_{i}-1}{a_{i}}=1-\frac{N+1}{a_{i}}$$ at the last iteration, as represented in Figure \ref{cross}.

\begin{figure}[H]
\centering
\begin{tikzpicture}
\path( 1,1)node[shape=circle,draw,label=above:$i$,inner sep=0] (pri) {$\otimes$}
     ( 0,0)node[shape=circle,draw,label=below:$-N$,fill=black] (pri1) {}
     ( 2,0)node[shape=circle,draw,label=below:$-N$,fill=black] (pri2) {}
     ( 4,0)node[shape=circle,draw,label=right:$-$,fill=black] (seg1) {}
     ( 4,1)node[shape=circle,draw,label=right:$-$,fill=black] (seg2) {}
     ( 6,1)node[shape=circle,draw,label=above:$i$,inner sep=0] (ter) {$\otimes$}
     ( 6,0)node[shape=circle,draw,label=right:$^{1-\frac{N+1}{a_{i}}}$,fill=black] (ter1) {};

      \draw[-](pri)--(pri1);
      \draw[-](pri)--(pri2);
      \draw[dotted](pri1)--(pri2);

      \draw[dotted](seg1)--(seg2);

      \draw[-](ter)--(ter1);

       \node[text width=6cm, anchor=west, right] at (2.5,0.5) {$\rightarrow$};

\end{tikzpicture}
\caption{\label{cross} Join.}
\end{figure}

And, following Lemma \ref{lem11}, for each $\cup$-node at depth $i$ of the cotree we perform $a_{i}-1$ iterations that leave negative permanent assignments $\frac{j+1}{j}(-N)$ for $j=1\ldots a_{i}-1$ by (\ref{cupperm}) and by (\ref{cuprem}) a remaining value $$d_{l}\leftarrow\frac{-N}{a_{i}}$$ at the last iteration, as depicted in Figure \ref{union}.

\begin{figure}[H]
\centering
\begin{tikzpicture}
\path( 1,1)node[shape=circle,draw,label=above:$i$,inner sep=0] (pri) {$\cup$}
     ( 0,0)node[shape=circle,draw,label=below:$-N$,fill=black] (pri1) {}
     ( 2,0)node[shape=circle,draw,label=below:$-N$,fill=black] (pri2) {}
     ( 4,0)node[shape=circle,draw,label=right:$-$,fill=black] (seg1) {}
     ( 4,1)node[shape=circle,draw,label=right:$-$,fill=black] (seg2) {}
     ( 6,1)node[shape=circle,draw,label=above:$i$,inner sep=0] (ter) {$\cup$}
     ( 6,0)node[shape=circle,draw,label=right:$\frac{-N}{a_{i}}$,fill=black] (ter1) {};

      \draw[-](pri)--(pri1);
      \draw[-](pri)--(pri2);
      \draw[dotted](pri1)--(pri2);

      \draw[dotted](seg1)--(seg2);

      \draw[-](ter)--(ter1);

       \node[text width=6cm, anchor=west, right] at (2.5,0.5) {$\rightarrow$};

\end{tikzpicture}
\caption{\label{union} Union.}
\end{figure}

We would like to emphasize that, as described in Example \ref{example1}, the vertices associated to permanent values are removed from the cotree. Therefore, after the specialization aforementioned in the terminal children of the cotree with identical value, we have $\sum_{i=1}^{r}(a_{i}-1)=n-r$ negative permanent assignments. And, after that we will continue processing the specialized cotree in Figure \ref{specialized} using Algorithm Diagonalize from left to right. Our strategy is to show that each $\cup$-node assigns a positive permanent value and each $\otimes$-node a negative permanent value with a positive remaining value. Consequently, it will imply that $N_{+}=|\cup|+1$ if $r$ is odd and $N_{+}=|\cup|$ if $r$ is even.

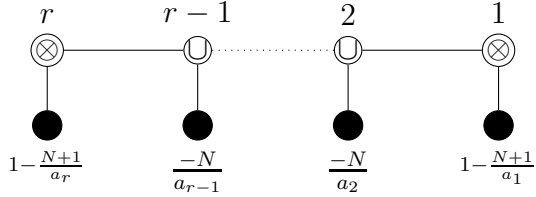
\begin{figure}[H]
\centering
\begin{tikzpicture}
\path( 0,1)node[shape=circle,draw,label=above:$r$,inner sep=0] (pri) {$\otimes$}
     ( 0,0)node[shape=circle,draw,label=below:$^{1-\frac{N+1}{a_{r}}}$,fill=black] (pri1) {}

     ( 2,1)node[shape=circle,draw,label=above:$r-1$,inner sep=0] (sec) {$\cup$}
     ( 2,0)node[shape=circle,draw,label=below:$\frac{-N}{a_{r-1}}$,fill=black] (sec1) {}

     ( 4,1)node[shape=circle,draw,label=above:$2$,inner sep=0] (ter) {$\cup$}
     ( 4,0)node[shape=circle,draw,label=below:$\frac{-N}{a_{2}}$,fill=black] (ter1) {}

     ( 6,1)node[shape=circle,draw,label=above:$1$,inner sep=0] (qua) {$\otimes$}
     ( 6,0)node[shape=circle,draw,label=below:$^{1-\frac{N+1}{a_{1}}}$,fill=black] (qua1) {};

      \draw[-](pri)--(pri1);

      \draw[-](sec)--(pri);
      \draw[-](sec)--(sec1);

      \draw[-](ter)--(ter1);

       \draw[dotted](ter)--(sec);

      \draw[-](qua)--(qua1);

       \draw[-](ter)--(qua);

\end{tikzpicture}
\caption{\label{specialized} Specialized Cotree.}
\end{figure}

It is important to point out that throughout this paper we execute Diagonalize  in two main steps. In the first one we process the leaves with identical value and obtain what we have called specialized cotree as illustrated in Figure \ref{specialized}. Then, we proceed with the algorithm on the specialized cotree as in Example \ref{example1}.

Consider the sequence of remaining values $\left\{ 1-\frac{N+1}{a_{r}}, \frac{-N}{a_{r-1}},\, ...,\, \frac{-N}{a_{2}},\, 1-\frac{N+1}{a_{1}}\right\}$ at the specialized cotree in Figure \ref{specialized} and the functions
\begin{align}
    f(X,Y)&=\frac{X\,Y-1}{X+Y-2}  \label{functioncross}\\
    g(X,Y)&=\frac{X\,Y}{X+Y}.      \label{functioncup}
\end{align}
 If we choose the initial value $s_{r+1}=  1-\frac{N+1}{a_{r}}$ and for each new value we pick $g$ for a union and $f$ for a join, then we obtain recursively
$s_{r}= g(s_{r+1},  \frac{-N}{a_{r-1}})$, $s_{r-1}= f(s_{r},  1-\frac{N+1}{a_{r-2}})$ and so on, until we reach $s_1$. These are the assignments we obtain by applying the \subia~and \subiia~from Diag$(T_{G},-N)$, which is performed from left to right. That is why we produce $s_i$ in the inverse order, from $s_{r+1}$ to $s_{1}$.

In the next two lemmas we show that by choosing the number of terminal children we can control the sign of the permanent and remaining assignments.

\begin{lemat}
\label{unionlema}
Suppose we have a $\cup$-node with leaves having values $s_{i+1}>0$ and $\frac{-N}{a_{i}}$ with $a_{i}>\frac{N}{s_{i+1}}$ as illustrated in Figure \ref{unionstep}. Then we use \subiia~ and obtain a positive permanent value $p_{i}$ and a negative remaining value $s_{i}$.
\end{lemat}
\begin{proof}

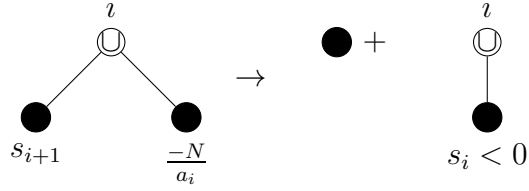
\begin{figure}[H]
\centering
\begin{tikzpicture}
\path( 1,1)node[shape=circle,draw,label=above:$i$,inner sep=0] (pri) {$\cup$}
     ( 0,0)node[shape=circle,draw,label=below:$s_{i+1}$,fill=black] (pri1) {}
     ( 2,0)node[shape=circle,draw,label=below:$\frac{-N}{a_{i}}$,fill=black] (pri2) {}
     ( 4,1)node[shape=circle,draw,label=right:$+$,fill=black] (seg2) {}

     ( 6,1)node[shape=circle,draw,label=above:$i$,inner sep=0] (ter) {$\cup$}
     ( 6,0)node[shape=circle,draw,label=below:$s_{i}<0$,fill=black] (ter1) {};

      \draw[-](pri)--(pri1);
      \draw[-](pri)--(pri2);

      \draw[-](ter)--(ter1);

       \node[text width=6cm, anchor=west, right] at (2.5,0.5) {$\rightarrow$};

\end{tikzpicture}
\caption{\label{unionstep} Union step.}
\end{figure}

Since $\alpha=s_{i+1}>\frac{N}{a_{i}}$ and $\beta=-\frac{N}{a_{i}}$, we can execute \subiia~because $\alpha+\beta>0$. It returns a positive permanent assignment $$p_{i}=d_{k}\leftarrow s_{i+1}-\frac{N}{a_{i}}>0,$$
and a negative remaining value $$s_{i}=d_{l}\leftarrow g\left(s_{i+1},-\frac{N}{a_{i}}\right)=\frac{s_{i+1}\left(\frac{-N}{a_{i}}\right)}{p_{i}}<0.$$

\end{proof}


\begin{lemat}
\label{crosslema}
Suppose we have a $\otimes$-node with leaves having values $s_{i+1}<0$ and $1-\frac{(N+1)}{a_{i}}$ with $a_{i}>\frac{N+1}{1-\left(\frac{1}{s_{i+1}}\right)}$ as illustraded in Figure \ref{crossstep}. Then we use \subia~ and obtain a negative permanent value $p_{i}$ and a positive remaining value $s_{i}$.
\end{lemat}

\begin{proof}

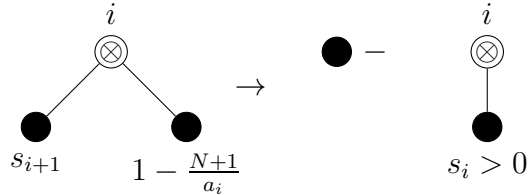
\begin{figure}[H]
\centering
\begin{tikzpicture}
\path( 1,1)node[shape=circle,draw,label=above:$i$,inner sep=0] (pri) {$\otimes$}
     ( 0,0)node[shape=circle,draw,label=below:$s_{i+1}$,fill=black] (pri1) {}
     ( 2,0)node[shape=circle,draw,label=below:$1-\frac{N+1}{a_{i}}$,fill=black] (pri2) {}

     ( 4,1)node[shape=circle,draw,label=right:$-$,fill=black] (seg2) {}

     ( 6,1)node[shape=circle,draw,label=above:$i$,inner sep=0] (ter) {$\otimes$}
     ( 6,0)node[shape=circle,draw,label=below:$s_{i}>0$,fill=black] (ter1) {};

      \draw[-](pri)--(pri1);
      \draw[-](pri)--(pri2);

      \draw[-](ter)--(ter1);

       \node[text width=6cm, anchor=west, right] at (2.5,0.5) {$\rightarrow$};

\end{tikzpicture}
\caption{\label{crossstep} Join step.}
\end{figure}

Since $\alpha=s_{i+1}<0$ and $\beta=1-\frac{N+1}{a_{i}}$ then we can execute \subia~because $\alpha+\beta-2=s_{i+1}+1-\frac{N+1}{a_{i}}-2=s_{i+1}-1-\frac{N+1}{a_{i}}<0$. It returns a negative permanent assignment $$p_{i}=d_{k}\leftarrow s_{i+1}-1-\frac{N+1}{a_{i}}<0,$$
and a positive remaining value
$$s_{i}=d_{l}\leftarrow f\left(s_{i+1},1-\frac{N+1}{a_{i}}\right)=\frac{s_{i+1}\left(1-\frac{N+1}{a_{i}}\right)-1}{p_{i}}>0$$
if and only if $$a_{i}>\frac{N+1}{1-\left(\frac{1}{s_{i+1}}\right)},$$
as our hypothesis holds.

\end{proof}

Next, we will show that the two lemmas above used together produce an algorithm to construct threshold graphs that are $(0,N]$-eigenvalue free.

\textbf{Initial step}: Considering $r$ odd, after the specialization, the leaf at the $\otimes$-node of depth $r$ has the assignment $$s_{r}=1-\frac{N+1}{a_{r}}>0$$ if and only if $a_{r}>N+1$. Then, using $a_{r}>N+1$ and Lemmas \ref{unionlema} and \ref{crosslema} we can construct a threshold graph of depth $r$ that is $(0,N]$-eigenvalue free.

And for $r$ even, after the specialization at the node of depth $r$ we have a leaf with assignment $$s_{r}=\frac{-N}{a_{r}}<0$$ for  $a_{r}\geq 2$. Then, using $a_{r}\geq 2$ and lemmas \ref{unionlema} and \ref{crosslema} we can construct a threshold graph of depth $r$ that is $(0,N]$-eigenvalue free.\\

The next theorem sums up the above observations.

\begin{teore}
\label{main_theorem1}
  Let $N>0$ be a fixed number and $r$ an odd number (case $r$ even is similar), if we choose natural numbers
  $a_{r}>N+1$, $a_{r-1}> \frac{N}{s_{r}}$, $a_{r-2}> \frac{N+1}{1-\left(\frac{1}{s_{r-1}}\right)}$, etc., then $T_{G}(a_{1},\, a_{2},\, \ldots,\, a_{r})$ is $(0,N]$-eigenvalue free.
\end{teore}

\begin{proof}
Considering the case $r$ odd we have to prove that $N_{+}=0_{+}=|\cup|+1$. We start the Diag$(T_{G},-N)$ using Lemmas \ref{lem10} and \ref{lem11} to process the terminal children with identical value. As afore-explained this initial process produces $(n-r)$ negative permanent diagonal assignments, adding $(n-r)$ to $N_{-}$. Then we proceed the Algorithm  Diagonalize on the specialized cotree in Figure \ref{specialized}. The last $\otimes$-node at depth $r$ has a leaf with assignment $s_{r}=1-\frac{N+1}{a_{r}}>0$ which will be relocated to the next level at the $\cup$-node as shown in Figure \ref{specialized1}.

\begin{figure}[H]
\centering
\begin{tikzpicture}
\path( 0,0)node[shape=circle,draw,label=below:$^{s_{r}=1-\frac{N+1}{a_{r}}}$,fill=black] (pri1) {}

     ( 2,1)node[shape=circle,draw,label=above:$r-1$,inner sep=0] (sec) {$\cup$}
     ( 2,0)node[shape=circle,draw,label=below:$\frac{-N}{a_{r-1}}$,fill=black] (sec1) {}

     ( 4,1)node[shape=circle,draw,label=above:$2$,inner sep=0] (ter) {$\cup$}
     ( 4,0)node[shape=circle,draw,label=below:$\frac{-N}{a_{2}}$,fill=black] (ter1) {}

     ( 6,1)node[shape=circle,draw,label=above:$1$,inner sep=0] (qua) {$\otimes$}
     ( 6,0)node[shape=circle,draw,label=below:$^{1-\frac{N+1}{a_{1}}}$,fill=black] (qua1) {};

      \draw[-](sec)--(pri1);
      \draw[-](sec)--(sec1);

      \draw[-](ter)--(ter1);

       \draw[dotted](ter)--(sec);

      \draw[-](qua)--(qua1);

       \draw[-](ter)--(qua);

\end{tikzpicture}
\caption{\label{specialized1} Specialized Cotree relocated.}
\end{figure}
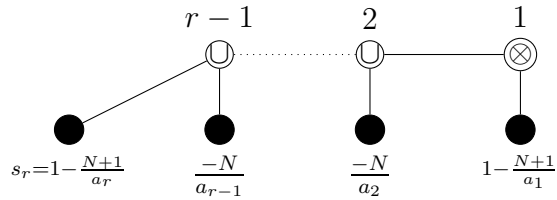

Using Lemma \ref{unionlema} at the $\cup$-node at depth $r-1$ we obtain a positive permanent diagonal value $p_{r-1}>0$ and a remaining assignment $s_{r-1}<0$. The leaf having the value $s_{r-1}$ is then relocated to the next $\otimes$-node at level $r-2$ as depicted in Figure \ref{specialized2}. At the $\otimes$-node at depth $r-2$ we apply Lemma \ref{crosslema} and it creates a negative permanent diagonal value $p_{r-2}<0$ and  a remaining value $s_{r-2}>0$ which will be relocated to the next level $r-3$.

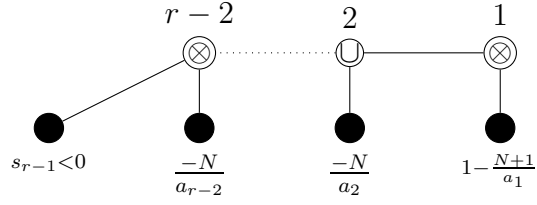
\begin{figure}[H]
\centering
\begin{tikzpicture}
\path( 0,0)node[shape=circle,draw,label=below:$^{s_{r-1}<0}$,fill=black] (pri1) {}

     ( 2,1)node[shape=circle,draw,label=above:$r-2$,inner sep=0] (sec) {$\otimes$}
     ( 2,0)node[shape=circle,draw,label=below:$\frac{-N}{a_{r-2}}$,fill=black] (sec1) {}

     ( 4,1)node[shape=circle,draw,label=above:$2$,inner sep=0] (ter) {$\cup$}
     ( 4,0)node[shape=circle,draw,label=below:$\frac{-N}{a_{2}}$,fill=black] (ter1) {}

     ( 6,1)node[shape=circle,draw,label=above:$1$,inner sep=0] (qua) {$\otimes$}
     ( 6,0)node[shape=circle,draw,label=below:$^{1-\frac{N+1}{a_{1}}}$,fill=black] (qua1) {};

      \draw[-](sec)--(pri1);
      \draw[-](sec)--(sec1);

      \draw[-](ter)--(ter1);

       \draw[dotted](ter)--(sec);

      \draw[-](qua)--(qua1);

       \draw[-](ter)--(qua);

\end{tikzpicture}
\caption{\label{specialized2} Specialized Cotree relocated.}
\end{figure}

Continuing this process, the $\otimes$-node at level $1$ will have two leafs with assignments $s_{2}$ and $1-\frac{N+1}{a_{1}}$ as illustrated in Figure \ref{finalcross} left. Once we apply Lemma \ref{crosslema} we process the two remaining vertices whose diagonal values will be $p_{1}<0$ and $s_{1}>0$.

\begin{figure}[H]
\centering
\begin{tikzpicture}
\path( 1,1)node[shape=circle,draw,label=above:$1$,inner sep=0] (pri) {$\otimes$}
     ( 0,0)node[shape=circle,draw,label=below:$s_{2}$,fill=black] (pri1) {}
     ( 2,0)node[shape=circle,draw,label=below:$1-\frac{N+1}{a_{1}}$,fill=black] (pri2) {}

     ( 4,1)node[shape=circle,draw,label=right:$-$,fill=black] (seg2) {}

     ( 6,1)node[shape=circle,draw,label=above:$1$,inner sep=0] (ter) {$\otimes$}
     ( 6,0)node[shape=circle,draw,label=below:$s_{1}>0$,fill=black] (ter1) {};

      \draw[-](pri)--(pri1);
      \draw[-](pri)--(pri2);

      \draw[-](ter)--(ter1);

       \node[text width=6cm, anchor=west, right] at (2.5,0.5) {$\rightarrow$};

\end{tikzpicture}
\caption{\label{finalcross} Last iteration.}
\end{figure}
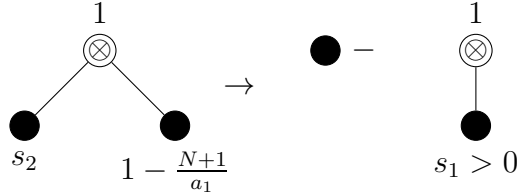

The algorithm stops, and then each $\cup$-node has produced a positive permanent diagonal assignment, adding $|\cup|$ to $N_{+}$. And each $\otimes$-node, except the one at depth $r$, has produced a negative permanent diagonal value, adding $|\otimes|-1$ to $N_{-}$. And the final permanent diagonal value is $s_{1}>0$, adding $+1$ to $N_{+}$. Therefore $N_{-}=(n-r)+|\otimes|-1$, $N_{0}=0$ and $N_{+}=|\cup|+1$.

\end{proof}

\begin{remark}
The values $a_{1},\, a_{2},\, \ldots,\, a_{r}$ are natural numbers so $a_{r-1}> \frac{N}{s_{r}}$ for $r$ odd (case $r$ even is similar), it means that $a_{r-1} \geq  \left\lceil\frac{N}{s_{r}}\right\rceil$ if $\frac{N}{s_{r}}$ is not an integer number, otherwise $a_{r-1} \geq  \left\lceil \frac{N}{s_{r}}\right\rceil +1$. To avoid this confusion we can always choose $a_{r}\geq 1+\left\lfloor N+1\right\rfloor$, $a_{r-1}\geq  1+\left\lfloor \frac{N}{s_{r}}\right\rfloor$, $a_{r-2}\geq  1+\left\lfloor \frac{N+1}{1-\left(\frac{1}{s_{r-1}}\right)}\right\rfloor$, etc., to construct  $T_{G}(a_{1},\, a_{2},\, \ldots,\, a_{r})$.
\end{remark}

Now we can define an algorithm to produce threshold graphs $T_{G}(a_{1},\, a_{2},\, \ldots,\, a_{r})$  satisfying Theorem \ref{main_theorem1}:\\

\begin{algorithm}[H]
\caption{The Right Free Interval Algorithm: RFI$(N,r)$}\label{alg1}
\begin{algorithmic}
\Require a positive real number $N$ and a positive integer $r$
\Ensure cotree $T_{G}(a_{1},a_{2},\ldots,a_{r})$
\If{$r$ is odd}
    \State Choose $a_{r}\geq 1+\lfloor N+1\rfloor$
    \State $s_{r} \gets1-\frac{1+N}{a_{r}}$  
\ElsIf{$r$ is even}
    \State  Choose $a_{r} \geq 2$
    \State $s_{r} \gets  -\frac{N}{a_{r}}$
\EndIf
\For{$i=r-1$ to $1$}

\If{$i$ is odd}
    \State Choose $a_{i}\geq 1+\left\lfloor \frac{N+1}{1-\left(\frac{1}{s_{i+1}}\right)}\right\rfloor$
    \State $p_{i} \gets s_{i+1}-1-\left(\frac{N+1}{a_{i}}\right) $
    \State $s_{i} \gets f\left(s_{i+1},1-\left(\frac{N+1}{a_{i}}\right)\right)$
\ElsIf{$i$ is even}
    \State Choose $a_{i}\geq 1+\left\lfloor \frac{N}{s_{i+1}} \right\rfloor$
    \State $p_{i} \gets s_{i+1}-\frac{N}{a_{i}} $
    \State $s_{i} \gets g\left(s_{i+1},-\frac{N}{a_{i}}\right)$
\EndIf

\EndFor

\end{algorithmic}
\end{algorithm}

\begin{defn}
Let $N>0$ be a fixed number and $r$ an odd number (case $r$ even is similar), if we choose natural numbers
$a_{r}=1+\left\lfloor N+1\right\rfloor$, $a_{r-1}=  1+\left\lfloor \frac{N}{s_{r+1}}\right\rfloor$, $a_{r-2}=  1+\left\lfloor \frac{N+1}{1-\left(\frac{1}{s_{r}}\right)}\right\rfloor$, etc., then $T_{G}(a_{1},\, a_{2},\, \ldots,\, a_{r})$ is called the initial threshold w.r.t. $N$, having no eigenvalues in the interval $(0,\, N]$.
\end{defn}

Notice that the initial threshold is obtained by making the smallest choice for each $a_{i}$ at {\rm RFI}$(N,r)$.\\

Next, we denote the first positive eigenvalue of a graph $G$ by a $+$ sign in the exponent, as $\theta^{+}(G)$.

\begin{Example}
Given $N=4.8$ and $r=5$ then {\rm RFI}$(4.8,\, 5)$ generates the initial threshold graph with cotree $T_{G}(5,145,5,145,6)$ which is $(0,4.8]$-eigenvalue free. Indeed, a direct computation shows that $\theta^{+}(G)=4.80000053517$.
\end{Example}

\begin{Example}
\label{rightex2}
Given $N=4.8$ and $r=6$ then {\rm RFI}$(4.8,\, 6)$  generates the initial threshold graph with cotree $T_{G}(6,44,6,36,5,2)$ which  is $(0,4.8]$-eigenvalue free. Indeed, a direct computation shows that $\theta^{+}(G)=4.80016011291$.
\end{Example}


\section{Interval $[M,-1)$}
\label{Inertia-1}

Given a real number $M<-1$, as before, our method consists in constructing a threshold graph $G$ with cotree $T_{G}(a_{1},\, a_{2},\, \ldots,\, a_{r})$, by choosing each $a_{i}\geq 1$ for $1\leq i\leq r-1$ and $a_{r}\geq 2$, and showing that this threshold graph $G$ is $[M,-1)$-eigenvalue free using the fact  that $M_{+}=(-1)_{0}+0_{0}+0_{+}$. Next we compute $(-1)_{0}+0_{0}+0_{+}$ for the cases $r$ odd and even.\\

Let $G$ be a threshold graph with cotree representation $T_{G}(a_{1},\, a_{2},\, \ldots,\, a_{r})$ with $a_{r}\geq 2$ and $r$ odd. Then $r=|\otimes|+|\cup|$ where $|\otimes|=|\cup|+1$. Using Theorem \ref{teo6} we obtain that $(-1)_{0}=\sum_{i=1}^{|\otimes|}(a_{2i-1}-1)$, so

\begin{equation*}
\begin{split}
(-1)_{0}+0_{0}+0_{+} & =\sum_{i=1}^{|\otimes|}(a_{2i-1}-1)+0_{0}+(n-0_{-}-0_{0})\\
                     & =\sum_{i=1}^{|\otimes|}(a_{2i-1}-1)+n-\left(\sum_{i=1}^{|\otimes|}(a_{2i-1})-1\right)\\
                     &=n-|\otimes|+1\\
                     & =(n-r)+(|\cup|+1),
\end{split}
\end{equation*}

since $r=|\cup|+|\otimes|$.

And, if $r$ is even then $0_{-}=\sum_{i=1}^{|\otimes|}(a_{2i-1})$. Hence,
$$(-1)_{0}+0_{0}+0_{+}=n-|\otimes|=(n-r)+|\cup|.$$

In this case, our approach is to construct an threshold graph such that $M_{+}=(n-r)+|\cup|+1$ if $r$ is odd and $M_{+}=(n-r)+|\cup|$ if $r$ is even.


\section{Diag$(T_{G},-M)$}
\label{secM}

We start the algorithm Diag$(T_{G},-M)$, for a fixed number $M<-1$, processing the leaves with identical value using Lemmas \ref{lem10} and \ref{lem11} as follows.\\

Lemma \ref{crosslema} is applied at each $\otimes$-node at depth $i$. We perform $a_{i}-1$ iterations that leave positive permanent diagonal values $\frac{j+1}{j}(-M-1)$ for $j=1\ldots a_{i}-1$ by (\ref{crossperm}). And, by (\ref{crossrem}), at the last iteration, a remaining assignment $$d_{l}\leftarrow\frac{-M+a_{i}-1}{a_{i}}=1-\frac{M+1}{a_{i}}$$ as represented in Figure \ref{crossleft}.

\begin{figure}[H]
\centering
\begin{tikzpicture}
\path( 1,1)node[shape=circle,draw,label=above:$i$,inner sep=0] (pri) {$\otimes$}
     ( 0,0)node[shape=circle,draw,label=below:$-M$,fill=black] (pri1) {}
     ( 2,0)node[shape=circle,draw,label=below:$-M$,fill=black] (pri2) {}
     ( 4,0)node[shape=circle,draw,label=right:$+$,fill=black] (seg1) {}
     ( 4,1)node[shape=circle,draw,label=right:$+$,fill=black] (seg2) {}
     ( 6,1)node[shape=circle,draw,label=above:$i$,inner sep=0] (ter) {$\otimes$}
     ( 6,0)node[shape=circle,draw,label=right:$^{1-\frac{M+1}{a_{i}}}$,fill=black] (ter1) {};

      \draw[-](pri)--(pri1);
      \draw[-](pri)--(pri2);
      \draw[dotted](pri1)--(pri2);

      \draw[dotted](seg1)--(seg2);

      \draw[-](ter)--(ter1);

       \node[text width=6cm, anchor=west, right] at (2.5,0.5) {$\rightarrow$};

\end{tikzpicture}
\caption{\label{crossleft} Join.}
\end{figure}

And, at each $\cup$-node at depth $i$ we use Lemma \ref{unionlema} as follows. We perform $a_{i}-1$ iterations that leave positive permanent diagonal values $\frac{j+1}{j}(-M)$ for $j=1\ldots a_{i}-1$ by (\ref{cupperm}). And a remaining value
$$d_{l}\leftarrow\frac{-M}{a_{i}}$$
at the last iteration by (\ref{cuprem}) as represented in Figure \ref{unionleft}.

\begin{figure}[H]
\centering
\begin{tikzpicture}
\path( 1,1)node[shape=circle,draw,label=above:$i$,inner sep=0] (pri) {$\cup$}
     ( 0,0)node[shape=circle,draw,label=below:$-M$,fill=black] (pri1) {}
     ( 2,0)node[shape=circle,draw,label=below:$-M$,fill=black] (pri2) {}
     ( 4,0)node[shape=circle,draw,label=right:$+$,fill=black] (seg1) {}
     ( 4,1)node[shape=circle,draw,label=right:$+$,fill=black] (seg2) {}
     ( 6,1)node[shape=circle,draw,label=above:$i$,inner sep=0] (ter) {$\cup$}
     ( 6,0)node[shape=circle,draw,label=right:$\frac{-M}{a_{i}}$,fill=black] (ter1) {};

      \draw[-](pri)--(pri1);
      \draw[-](pri)--(pri2);
      \draw[dotted](pri1)--(pri2);

      \draw[dotted](seg1)--(seg2);

      \draw[-](ter)--(ter1);

       \node[text width=6cm, anchor=west, right] at (2.5,0.5) {$\rightarrow$};

\end{tikzpicture}
\caption{\label{unionleft} Union.}
\end{figure}

Therefore, after the specialization in the leaves with identical value we already have $\sum_{i=1}^{r}(a_{i}-1)=n-r$ positive permanent diagonal values. Hence, following our strategy we will work with the cotree represented in Figure \ref{specializedleft}  and we will show that each $\cup$-node gives a positive permanent diagonal value and each $\otimes$-node returns a negative permanent diagonal value with a positive remaining assignment. It will imply that $M_{+}=(n-r)+(|\cup|+1)$ if $r$ is odd and $M_{+}=(n-r)+|\cup|$ if $r$ is even.

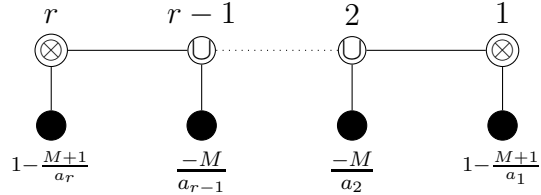
\begin{figure}[H]
\centering
\begin{tikzpicture}
\path( 0,1)node[shape=circle,draw,label=above:$r$,inner sep=0] (pri) {$\otimes$}
     ( 0,0)node[shape=circle,draw,label=below:$^{1-\frac{M+1}{a_{r}}}$,fill=black] (pri1) {}

     ( 2,1)node[shape=circle,draw,label=above:$r-1$,inner sep=0] (sec) {$\cup$}
     ( 2,0)node[shape=circle,draw,label=below:$\frac{-M}{a_{r-1}}$,fill=black] (sec1) {}

     ( 4,1)node[shape=circle,draw,label=above:$2$,inner sep=0] (ter) {$\cup$}
     ( 4,0)node[shape=circle,draw,label=below:$\frac{-M}{a_{2}}$,fill=black] (ter1) {}

     ( 6,1)node[shape=circle,draw,label=above:$1$,inner sep=0] (qua) {$\otimes$}
     ( 6,0)node[shape=circle,draw,label=below:$^{1-\frac{M+1}{a_{1}}}$,fill=black] (qua1) {};

      \draw[-](pri)--(pri1);

      \draw[-](sec)--(pri);
      \draw[-](sec)--(sec1);

      \draw[-](ter)--(ter1);

       \draw[dotted](ter)--(sec);

      \draw[-](qua)--(qua1);

       \draw[-](ter)--(qua);

\end{tikzpicture}
\caption{\label{specializedleft} Specialized Cotree.}
\end{figure}

Next, we prove that each $\otimes$-node leaves a permanent negative value and a remaining value positive. Whereas, each $\cup$-node leaves a permanent positive value. And, these results will prove that we obtain $|\cup|$ positive permanent values. In the last iteration, at the final $\otimes$-node, we will have a positive remaining diagonal assignment . It will imply that $M_{+}=(n-r)+(|\cup|+1)$ if $r$ is odd and $M_{+}=(n-r)+|\cup|$ if $r$ is even.

Consider the sequence of remaining assignments $\left\{ 1-\frac{M+1}{a_{r}}, \frac{-M}{a_{r-1}},\, ...,\, \frac{-M}{a_{2}},\, 1-\frac{M+1}{a_{1}}\right\}$ at the specialized cotree in Figure \ref{specializedleft}. If we choose the initial value $s_{r+1}=  1-\frac{M+1}{a_{r}}$ and for each iteration we choose the function (\ref{functioncup}) for a union and (\ref{functioncross}) for a join, then we obtain recursively $s_{r}= g\left(s_{r+1},  \frac{-M}{a_{r-1}}\right)$, $s_{r-1}= f\left(s_{r},  1-\frac{N+1}{a_{r-2}}\right)$ and so on, until we reach $s_1$. These are the assignments we obtain by applying the \subia~or \subiia~from Diag$(T_{G},-M)$, which is performed from left to right. That is why we produce $s_i$ in the inverse order, from $s_{r+1}$ to $s_{1}$.

In the next two lemmas we show that by choosing the number of terminal children we can control the sign of the permanent and remaining assignments.

\begin{lemat}
\label{unionlemaleft}
Suppose we have a $|\cup|$-node with leaves having values $s_{i+1}>1$ and $\frac{-M}{a_{i}}$ with $a_{i}>-M+\frac{M}{s_{i+1}}$ as illustrated in Figure \ref{unionstepleft}. Then we use \subiia~ and obtain a positive permanent value $p_{i}$ and a remaining value $0<s_{i}<1$.
\end{lemat}

\begin{proof}

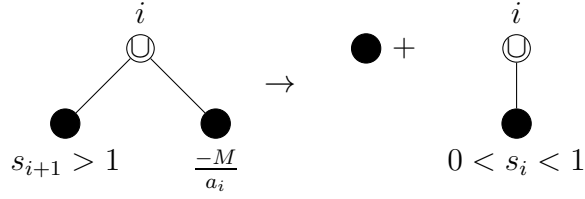
\begin{figure}[H]
\centering
\begin{tikzpicture}
\path( 1,1)node[shape=circle,draw,label=above:$i$,inner sep=0] (pri) {$\cup$}
     ( 0,0)node[shape=circle,draw,label=below:$s_{i+1}>1$,fill=black] (pri1) {}
     ( 2,0)node[shape=circle,draw,label=below:$\frac{-M}{a_{i}}$,fill=black] (pri2) {}

     ( 4,1)node[shape=circle,draw,label=right:$+$,fill=black] (seg2) {}

     ( 6,1)node[shape=circle,draw,label=above:$i$,inner sep=0] (ter) {$\cup$}
     ( 6,0)node[shape=circle,draw,label=below:$0<s_{i}<1$,fill=black] (ter1) {};

      \draw[-](pri)--(pri1);
      \draw[-](pri)--(pri2);

      \draw[-](ter)--(ter1);

       \node[text width=6cm, anchor=west, right] at (2.5,0.5) {$\rightarrow$};

\end{tikzpicture}
\caption{\label{unionstepleft} Union step.}
\end{figure}
Since $\alpha=s_{i+1}>1$ and $\beta=-\frac{M}{a_{i}}$ then $\alpha+\beta>1$ and we can use \subiia~. It returns a permanent diagonal assignment
$$p_{i}=d_{k}\leftarrow s_{i+1}-\frac{M}{a_{i}}>1$$ and a remaining one $$s_{i}=d_{l}\leftarrow g\left(s_{i+1},-\frac{M}{a_{i}}\right)=\frac{\left(s_{i+1}\right)\left(\frac{-M}{a_{i}}\right)}{p_{i}}>0.$$
Now, we want to show that $s_{i}<1$, so
$$s_{i}=\frac{\left(s_{i+1}\right)\left(\frac{-M}{a_{i}}\right)}{s_{i+1}-\frac{M}{a_{i}}}<1$$
if and only if
$$a_{i}>-M+\frac{M}{s_{i+1}},$$
as our hypothesis holds.

\end{proof}

\begin{lemat}
\label{crosslemaleft}
Suppose we have a $\otimes$-node with leaves having assignments $0<s_{i+1}<1$ and $1-\frac{(M+1)}{a_{i}}$ with $a_{i}>-\frac{M+1}{1-s_{i+1}}$ as illustrated in Figure \ref{crossstepleft}. Then we use \subia~ and obtain a negative permanent value $p_{i}$ and a positive remaining value $s_{i}>1$.
\end{lemat}

\begin{proof}

\begin{figure}[H]
\centering
\begin{tikzpicture}
\path( 1,1)node[shape=circle,draw,label=above:$i$,inner sep=0] (pri) {$\otimes$}
     ( 0,0)node[shape=circle,draw,label=below:$s_{i+1}$,fill=black] (pri1) {}
     ( 2,0)node[shape=circle,draw,label=below:$1-\frac{M+1}{a_{i}}$,fill=black] (pri2) {}

     ( 4,1)node[shape=circle,draw,label=right:$-$,fill=black] (seg2) {}

     ( 6,1)node[shape=circle,draw,label=above:$i$,inner sep=0] (ter) {$\otimes$}
     ( 6,0)node[shape=circle,draw,label=below:$s_{i}>1$,fill=black] (ter1) {};

      \draw[-](pri)--(pri1);
      \draw[-](pri)--(pri2);


      \draw[-](ter)--(ter1);

       \node[text width=6cm, anchor=west, right] at (2.5,0.5) {$\rightarrow$};

\end{tikzpicture}
\caption{\label{crossstepleft} Join step.}
\end{figure}
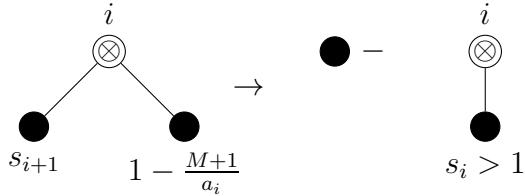

Since $\alpha=s_{i+1}$ and $\beta=1-\frac{M+1}{a_{i}}$ then $\alpha+\beta-2=s_{i+1}+1-\frac{M+1}{a_{i}}-2=s_{i+1}-1-\frac{M+1}{a_{i}}<0$
if and only if $a_{i}>-\frac{M+1}{1-s_{i+1}}.$ Therefore, we can use \subia~ and the permanent diagonal assignment will be
$$p_{i}=d_{k}\leftarrow s_{i+1}-1-\frac{M+1}{a_{i}}<0.$$
And the remaining one
$$s_{i}=d_{l}\leftarrow f\left(s_{i+1},1-\frac{M+1}{a_{i}}\right)=\frac{(s_{i+1})(1-\frac{M+1}{a_{i}})-1}{s_{i+1}-1-\frac{M+1}{a_{i}}}>1$$
if and only if
$$s_{i+1}(-(M+1))<-(M+1),$$
which is satisfied if and only if $s_{i+1}<1$.

\end{proof}

As in Section \ref{secN}, the two lemmas above used together produce an algorithm to construct threshold graphs that are $[M,-1)$-eigenvalue free.

\textbf{Initial step}: Considering $r$ odd, after the specialization of the leaves with identical value, the $\otimes$-node at depth $r$ has a leaf with assignment $$s_{r}=1-\frac{M+1}{a_{r}}>1$$ if and only if
$\frac{-(M+1)}{a_{r}}>0$, which is satisfied for all $a_{r}\geq 2$.
 Then, using $a_{r}\geq 2$ and Lemmas \ref{unionlemaleft} and \ref{crosslemaleft} we can construct a threshold graph of depth $r$ that is $[M,-1)$-eigenvalue free.
And, for $r$ even, after the specialization, the $\cup$-node at depth $r$ has a leaf with assignment $$s_{r}=\frac{-M}{a_{r}}.$$
$s_{r}>0$ is trivially satisfied and $s_{r}<1$ iff $a_{r}>-M$. Then, using $a_{r}>-M$ and Lemmas \ref{unionlemaleft} and \ref{crosslemaleft} we can construct a threshold graph of depth $r$ that is $[M,-1)$-eigenvalue free.\\

\begin{remark}
In the left case we have $a_{r}$ free if $r$ is odd and in the right case we have $a_{r}$ free if $r$ is even.
\end{remark}

The next theorem sums up the above results and its proof is similar to the proof of Theorem \ref{main_theorem1}.

\begin{teore}
\label{main_theorem2}
  Let $M<-1$ be a fixed number and $r$ an odd number (the even case is similar), if we choose natural numbers
  $a_{r} \geq 2$, $a_{r-1}> -M\left(1-\frac{1}{s_{r}}\right)$, $a_{r-2}> \frac{M+1}{1-s_{r-1}}$, etc., then $T_{G}(a_{1},\, a_{2},\, \ldots,\, a_{r})$ is $[M,\, -1)$-eigenvalue free.
\end{teore}

Now we can define an algorithm to produce threshold graphs $T_{G}(a_{1},\, a_{2},\, \ldots,\, a_{r})$  satisfying Theorem \ref{main_theorem2}:\\

\begin{algorithm}[H]
\caption{The Left Free Interval Algorithm: LFI$(N,r)$}\label{alg2}
\begin{algorithmic}
\Require a negative real number $M<-1$ and a positive integer $r$
\Ensure cotree $T_{G}(a_{1},a_{2},\ldots,a_{r})$
\If{$r$ is odd}
    \State Choose $a_{r}\geq 2$
    \State $s_{r} \gets 1-\frac{1+M}{a_{r}}$  
\ElsIf{$r$ is even}
    \State  Choose $a_{r} \geq 1+\left\lfloor -M \right\rfloor$
    \State $s_{r} \gets  -\frac{M}{a_{r}}$
\EndIf
\For{$i=r-1$ to $1$}

\If{$i$ is odd}
    \State Choose $a_{i}\geq 1+\left\lfloor \frac{-M-1}{1-s_{i+1}}\right\rfloor$
    \State $p_{i} \gets s_{i+1}-1-\left(\frac{M+1}{a_{i}}\right) $
    \State $s_{i} \gets f\left(s_{i+1},1-\left(\frac{M+1}{a_{i}}\right)\right)$
\ElsIf{$i$ is even}
    \State Choose $a_{i}\geq 1+ \left\lfloor -M+\frac{M}{s_{i+1}} \right\rfloor$
    \State $p_{i} \gets s_{i+1}-\frac{M}{a_{i}} $
    \State $s_{i} \gets g\left(s_{i+1},-\frac{M}{a_{i}}\right)$
\EndIf

\EndFor

\end{algorithmic}
\end{algorithm}
As before, the initial threshold is obtained by making the smallest choice for each $a_{i}$ in {\rm LFI}$(M,r)$.\\

Next, we denote the first negative eigenvalue of a graph $G$ smaller than $-1$ by a $-$ sign in the exponent, as $\theta^{-}(G)$.

\begin{Example}
\label{counter}
Given $M=-3.3$ and $r=7$ then {\rm LFI}$(-3.3,\, 7)$  generates the initial threshold graph with cotree $T_{G}( 11,4,46,3,,35,2,2)$ which is $[-3.3, -1)$-eigenvalue free. Indeed, a direct computation shows that the maximum negative eigenvalue is $\theta^-(G)=-3.30000464177< M$.
\end{Example}

Curiously, the left  initial threshold graph and the right  initial threshold graph coincides in {\rm LFI}$(M,\, 7)$={\rm RFI}$(N,\, 7)$ = $T_{G}( 1,1,1,1,1,1,2)$ for $M=\frac{-1-\sqrt{2}}{2} <-1$ and $N=-1-M=\frac{-1+\sqrt{2}}{2}>0$. A direct computation shows that $\theta^-(G)=-1.24338010982 < M=-1.20710678118$ and $\theta^+(G)=0.231890667597 > N=0.20710678118$.
Notice that $T_{G}( 1,1,1,1,1,1,2)$ is the cotree associated to an anti-chain graph of order $n=8$.
However, as we exemplify below, this correspondence between $N$ and $M=-N-1$ does not hold for any threshold.

Using Example \ref{counter}, where $G$ is the threshold graph with associated cotree $T_{G}( 11,4,$ $46,$ $3,35,2,2)$, we compute $\theta^+(G)=0.558865493736 < N=-1-M= 2.3$. Actually, the right initial threshold graph obtained from RFI$(-3.3,\, 7)$ is
$T_{G}( 3,14,3,16,3,14,4)$ which satisfies $\theta^+(G)=2.30004052499> 2.3$.

\section{Infinite families}
\label{Family}

In this section, we show that the initial threshold graph generated by Algorithms \ref{alg1} or \ref{alg2} is an initial threshold graph for a family of threshold graphs that have the same $I$-eigenvalue free property.

We recall below the known interlacing property that can be found in \cite{hall2009interlacing}.

\begin{teore}
\label{interlacing}
Let $G$ be a graph and $H=G-v$, where $v$ is a vertex of $G$. If
$\lambda_{1}\geq\lambda_{2}\geq\cdots\geq\lambda_{n}$ and $\theta_{1}\geq\theta_{2}\geq\cdots\geq\theta_{n-1}$ are the eigenvalues of $A(G)$ and $A(H)$, respectively, then
$$\lambda_{i}\geq\theta_{i}\geq\lambda_{i+1}\mbox{ for each }i=1,2,\ldots,n-1.$$
\end{teore}

Given a threshold graph $G$ with eigenvalues $\lambda_{1}\geq\lambda_{2}\cdots\geq\lambda_{n}$ and $H=G-v$ with eigenvalues $\theta_{1}\geq\theta_{2}\cdots\geq\theta_{n-1}$.

Consider the cotree $T_{H}(a_{1},\ldots,a_{r})$ of depth $r$ associated to the threshold graph $H$ with $a_{i}\geq 1$ for $1\leq i\leq r-1$ and $a_{r}\geq 2$. Denote $m(0,H)=p$ and $m(-1,H)=q$ the respective multiplicities of $0$ and $-1$ in $H$. We recall that
$$\theta^{+}(H)=\mbox{ first positive eigenvalue of $H$ greater than $0$,}$$
and
$$\theta^{-}(H)=\mbox{ first negative eigenvalue of $H$
smaller than $-1$.}$$

Suppose that $M<\theta^{-}(H)$ and $\theta^{+}(H)>N$.
Using the notation above we have the following:
$$\theta_{n-1}\leq\cdots\leq\theta_{i+p+q+2}\leq\theta_{i+p+q+1}<M<-1=
\underbrace{\theta_{i+p+q}=\cdots=\theta_{i+p+2}=\theta_{i+p+1}}_{q=m(-1,H)}<0$$

$$0=\underbrace{\theta_{i+p}=\cdots=\theta_{i+2}=\theta_{i+1}}_{p=m(0,H)}<N<\theta_{i}\leq\cdots\theta_{2}\leq\theta_{1}.$$

Hence, $\theta_{i+p+q+1}=\theta^{-}<M$ and $\theta_{i}=\theta^{+}>N$.

Notice that, when we add a vertex $v$ to the threshold $H$ then we have two possibilities:
\begin{enumerate}
	\item we add a vertex $v$ to a $\cup$-node at depth $l$ that has $a_{l}$ terminal vertices in $H$. Then $G$ will have a $\cup$-node at depth $l$ that has $a_{l}+1$ terminal vertices;
	\item we add a vertex $v$ to a $\otimes$-node at depth $l$ that has $a_{l}$ terminal vertices in $H$. Then $G$ will have a $\otimes$-node at depth $l$ that has $a_{l}+1$ terminal vertices.
\end{enumerate}

Using Theorems \ref{teo5} and \ref{teo6} we have the following:

\begin{enumerate}
	\item $m(0,G)=m(0,H)+1$ and $m(-1,G)=m(-1,H)$;
	\item $m(0,G)=m(0,H)$ and $m(-1,G)=m(-1,H)+1$;
\end{enumerate}

Now, we apply the interlacing Theorem \ref{interlacing}.

$$\lambda_{i}\geq\underbrace{\theta_{i}}_{\theta^{+}}\geq\underbrace{\lambda_{i+1}}_{?}\geq\theta_{i+1}=\lambda_{i+2}=\cdots=\lambda_{i+p}=\theta_{i+p}=0\geq \underbrace{\lambda_{i+p+1}}_{-1 \mbox{ or }0}\geq$$

$$\geq\theta_{i+p+1}=-1=\lambda_{i+p+2}=\cdots=\lambda_{i+p+q}=\theta_{i+p+q}=-1\geq\underbrace{\lambda_{i+p+q+1}}_{?}\geq\underbrace{\theta_{i+p+q+1}}_{\theta^{-}}\geq\lambda_{i+p+q+2}.$$

Considering the two possible cases we conclude the following.
\begin{enumerate}
	\item  $m(0,G)=p+1$ and $m(-1,G)=q$. It implies that $\lambda_{i+1}=\lambda_{i+p+1}=0$ and $\lambda_{i+p+q+1}=-1$.
	\item $m(0,G)=p$ and $m(-1,G)=q+1$. It implies that $\lambda_{i+1}=0$ and $\lambda_{i+p+1}=\lambda_{i+p+q+1}=-1$.
\end{enumerate}

And, it implies that $$\lambda_{i}=\lambda^{+}(G)\geq\theta_{i}=\theta^{+}>N \mbox{ and } \lambda_{i+p+q+2}=\lambda^{-}(G)\leq\theta_{i+p+q+1}=\theta^{-}<M.$$

The above computations motivate the following definition.

\begin{defn}
Let $T_{G}(a_{1},\ldots,a_{r})$ with $a_{i}\geq 1$ for $1\leq i\leq r-1$ and $a_{r}\geq 2$, and $T_{G^{'}}(b_{1},b_{2},\ldots,b_{r})$ be two threshold graphs. We say that  $G \preceq G^{'}$ if $a_{i}\leq b_{i}$ for $i=1,2,\ldots,r$ i.e., if $G^{'}$ is generated from $G$ by adding any amount of leaves to $\cup$-nodes or $\otimes$-nodes (of course, $G \prec G^{'}$ if $G \preceq G^{'}$ and $G \neq G^{'}$). Notice that $G$ and $G^{'}$ have the same depth.
\end{defn}

Then, we have proved the next result.

\begin{coro}
\label{free_int_increasing}
For $I=(0,N]$ or $I=[M,-1)$ we have that,
if $T_{G}(a_{1},\ldots,a_{r})$ is $I$-eigenvalue free, then $T_{G^{'}}(b_{1},\ldots,b_{r})$ is $I$-eigenvalue free.
\end{coro}

In Example \ref{rightex2}, given $N=4.8$ and $r=6$, RFI$(4.8,6)$ generated the initial threshold graph $G$ with associated cotree $T_{G}(6,44,6,36,5,2)$ which is $(0,4.8]$-eigenvalue free. By Corollary \ref{free_int_increasing}, any threshold graph represented by the cotree $T_{G^{'}}(6+i_{1},44+i_{2},6+i_{3},36+i_{4},5+i_{5},2+i_{6})$ for any integers $i_{j}\geq 0$, $1\leq j\leq 6$, is $(0,4.8]$-eigenvalue free.

\section{Revisiting Ghorbani's work}
\label{secrec}

In this section we use our approach to revisit the conjecture proposed by Aguilar et al. in  \cite{aguilar2018spectral} proved  by Ghorbani in \cite{ghorbani2019eigenvalue}.

\begin{teore}
\label{right}
Given $N=\frac{-1+\sqrt{2}}{2} = 0.20710678+$ and $r$ any odd natural number (the even case is identical) then the algorithm RFI$(N,r)$ generates the initial threshold graph with cotree $T_{G}(1,...,1,2)$ (with $r-1$ leafs 1's) that is $(0,N]$-eigenvalue free. As $T_{G}(1,...,1,2)$ has less vertices than any other threshold graph $T_{G'}(a_{1},\, a_{2},\, \ldots,\, a_{r})$ with $a_{r}\geq 2$, $a_{i}\geq 1$ for $1\leq i\leq r-1$, and $r$ odd, by Corollary \ref{free_int_increasing} we conclude that there is no threshold graph with eigenvalues in the interval $(0,N]$.
\end{teore}

\begin{proof}
A direct computation shows that $a_{r}\geq 1+\lfloor N+1\rfloor=2$ so the initial choice is $a_{r}=2$. From there our proof is by induction. Also by direct computation, we can compute  $s_{r}=3/4-1/4\,\sqrt{2} \simeq 0.3964>0$ and $a_{r-1} \geq  1+\lfloor \frac{N}{s_{r}}\rfloor=1+ \lfloor 0.5224 \rfloor=1$, so the initial choice is $a_{r-1}=1$ and compute $s_{r-1}=g\left(s_{r},  \frac{-N}{a_{r-1}}\right)= g\left(s_{r},-N\right)= {\frac { \left( -3+\sqrt {2} \right)  \left( -1+\sqrt {2}
 \right) }{10-6\,\sqrt {2}}}
\simeq -0.4336<0$. Analogously, $a_{r-2}\geq  1+\left\lfloor \frac{N+1}{1-\left(\frac{1}{s_{r-1}}\right)}\right\rfloor= 1+\lfloor 0.3651 \rfloor=1$, is the initial choice, and compute $s_{r-2}=f\left(s_{r-1},  1-\frac{N+1}{a_{r-2}}\right)= f(g(s_{r},-N),-N)$.  In what follows we obtain the initial choice equal to 1 and therefore $s_{2k}=f\left(g\left(s_{2k+2},-N\right),-N\right)$, $s_{2k+1}=g\left(f\left(s_{2k+3},-N\right),-N\right)$. Now we will consider $k+1$. Both recurrences are explicit and easy to solve. For even indices we get $s_{2k}=f(g(s_{2k+2},-N),-N)=\varphi(s_{2k+2})$ where $$\varphi(t)={\frac { \left( 2\,\sqrt {2}+1 \right) t+2-2\,\sqrt {2}}{4\, \left( 1+
\sqrt {2} \right) t-2\,\sqrt {2}+1}}.$$ It is easy to see that $\varphi(t)=t$ has a unique solution $\mu={\frac {\sqrt {2}}{2(1+\sqrt {2})}} \simeq 0.2928$.
Since $\varphi'(t)<1$ for $t>\mu= 0.2928$(see Figure~\ref{fig:_varphi_psi}, left) and $s_0= 0.3964$ we get that if $\mu< s_{2k} < s_0$ then $s_{2k}=\varphi(s_{2k+2}) \in (\mu, s_{r})$ meaning that $0=\left\lfloor \frac{N}{\mu}\right\rfloor > \left\lfloor \frac{N}{s_{2k+2}}\right\rfloor>\left\lfloor \frac{N}{s_{r}}\right\rfloor=0 $, that is the initial choice for $2k+2$ is 1.\\
For the odd indices we get $s_{2k+1}=g(f(s_{2k+3},-N),-N)=\psi(s_{2k+3})$ where $$\psi(t)={\frac { \left( 1-\sqrt {2} \right)  \left( 2+ \left( -1+\sqrt {2}
 \right) t \right) }{ \left( 4\,\sqrt {2}-4 \right) t+5-2\,\sqrt {2}}}.$$  It is easy to see that $\psi(t)=t$ has a unique solution $\mu'={\frac {\sqrt {2}-2}{2\,\sqrt {2}-2}} \sim -0.7071$.
 Since $\psi'(t)<1$ for $t>\mu'= -0.7071$ (see Figure~\ref{fig:_varphi_psi}, right) and $s_{r-1}= -0.4336<0$ we get that, if $\mu'< s_{2k+3} < s_{r-1}$ then $s_{2k+1}=\psi(s_{2k+3}) \in (\mu', s_{r-1})$ meaning that $0=\left\lfloor \frac{N+1}{1-\left(\frac{1}{\mu}\right)}\right\rfloor > \left\lfloor \frac{N+1}{1-\left(\frac{1}{s_{2k+1}}\right)}\right\rfloor>\left\lfloor \frac{N+1}{1-\left(\frac{1}{s_{r-1}}\right)}\right\rfloor=0 $, which is the initial choice for $2k+1$ is also 1.\\

 In both cases we will always have $a_i=1$ as the initial choice. Given the decreasing convergence towards to the fixed points $\mu$ and $\mu'$ we deduce that our claim is true for a arbitrary large $r$, concluding our proof.
\begin{figure}
    \centering
    \includegraphics[width=6cm]{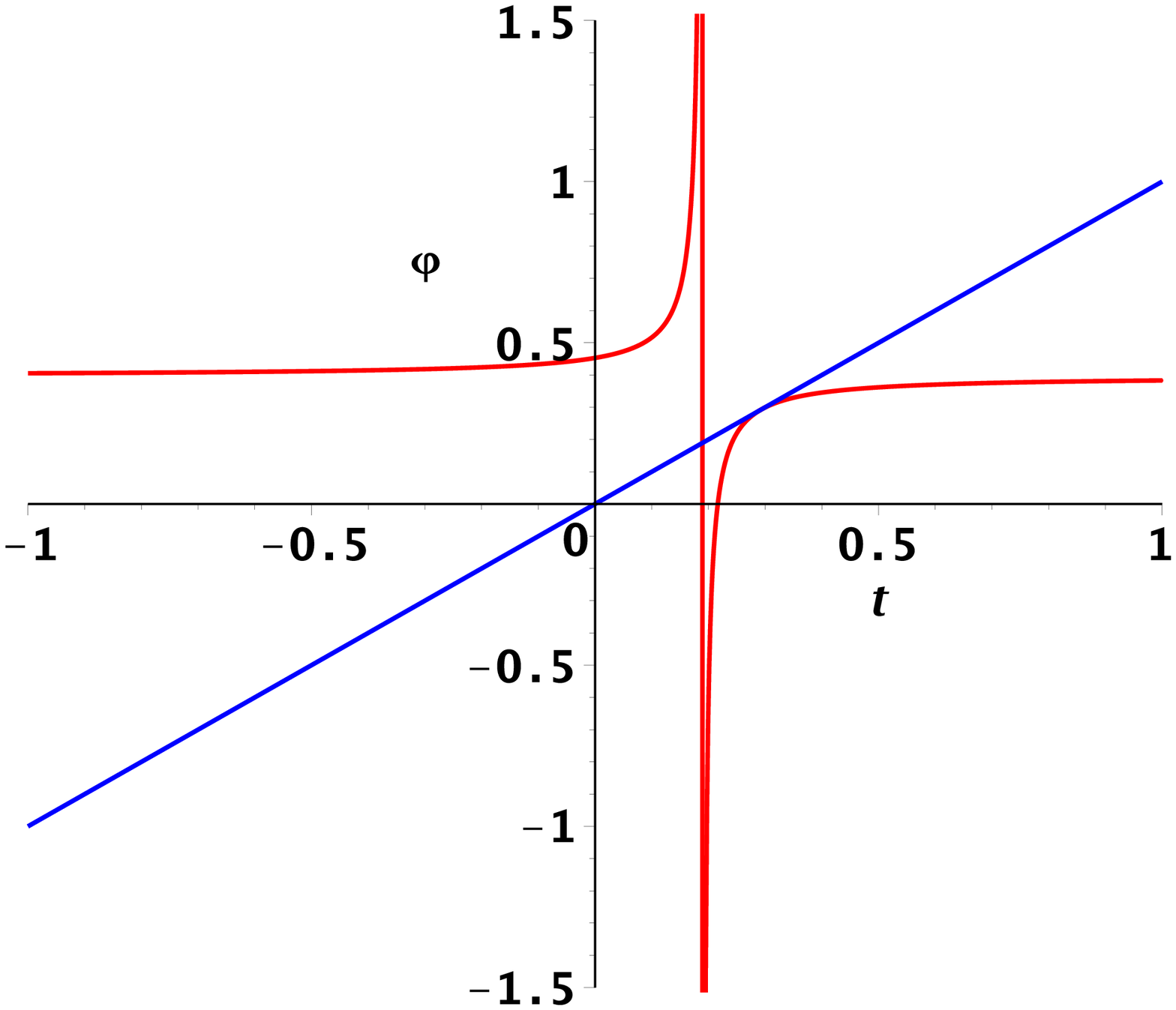}\;
    \includegraphics[width=6cm]{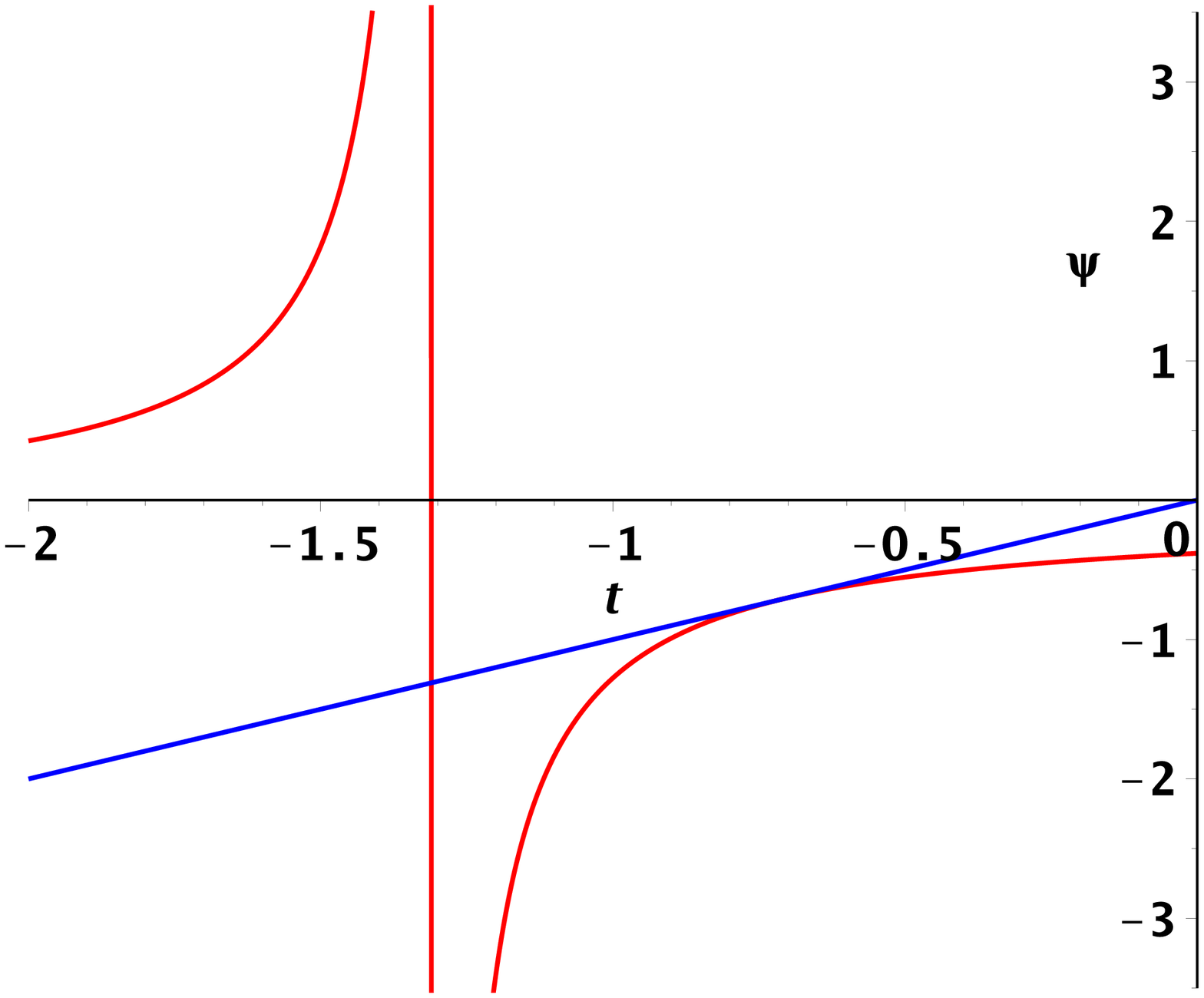}\;
    \caption{From the left, $\varphi(t)$ and $\psi(t)$}
    \label{fig:_varphi_psi}
\end{figure}
\end{proof}

The proof of the next result is similar to the one above.

\begin{teore}
\label{left}
Given $M=\frac{-1-\sqrt{2}}{2} = -1.207106781$ and $r$ any odd natural number (the even case is identical) then the algorithm LFI$(N,r)$ generates the initial threshold graph with cotree $T_{G}(1,...,1,2)$ (with $r-1$ leafs 1's) that is $(M,-1]$-eigenvalue free. As $T_{G}(1,...,1,2)$ has less vertices than any other threshold graph $T_{G'}(a_{1},\, a_{2},\, \ldots,\, a_{r})$ with $a_{r}\geq 2$, $a_{i}\geq 1$ for $1\leq i\leq r-1$, and $r$ odd, by Corollary \ref{free_int_increasing} we conclude that there is no threshold graph with eigenvalues in the interval $(0,N]$.
\end{teore}

\section{Future work}
\label{Future}

As a sequel of this work, we would like to investigate additional properties of the initial threshold graphs generated by Algorithms RFI$(N,r)$ and LFI$(N,r)$, such as the existence of minimal ones, w.r.t. the $I$-eigenvalue free property.

By Corollary~\ref{free_int_increasing}, the relation $G \preceq G^{'}$ preserves the $I$-eigenvalue free property. In this way, given for instance $I=(0,N]$ and the initial threshold graph $T_{G}(a_{1},\ldots,a_{r})$, obtained from RFI$(N,r)$, we already know that, if $G \preceq G^{'}$ then $G^{'}$ is also $(0,N]$-eigenvalue free. However, if $G^{''} \preceq G$ the situation is not clear. By a simple counting procedure we can see that there are exactly $K=a_{1}\cdot\ldots\cdot a_{r}$ of such graphs. For a fairly small number $N=\sqrt{5}=2.26+$ the initial threshold graph, obtained from RFI$(\sqrt{5},7)$, is $T_{G}(4,10,4,55,3,12,4)$, producing $K=4\cdot 10\cdot 4\cdot 55\cdot 3\cdot 12\cdot 4= 1,267,200$ threshold graphs smaller than itself to compare.   We can perform an exhaustive search for a $G^{''} \preceq G$ such that $G^{''}$ is also $(0,N]$-eigenvalue free, by directly computing the spectrum of each graph. This task could be extremely hard from a computational point of view but we can make it easy by using the Diagonalize algorithm (it took about 5 minutes for $N=\sqrt{5}$ using a naive implementation). Comparing the number of positive outputs for Diag$\left(T_{G''}, \frac{-1+\sqrt{2}}{2}\right)$ (which has no zero outputs, because $\frac{-1+\sqrt{2}}{2}$ is not eigenvalue for any threshold graph) and Diag$(T_{G''}, N)$ (which should not have zero outputs if $N < \theta^+(G'')$) we can see that $G^{''}$ is also $(0,N]$-eigenvalue free if, and only if, this number of positive outputs remains unchanged. Remembering that, after the specialization, $n-r$ outputs will be negative, we only need compute $r$ values in Diag$\left(T_{G''}, \frac{-1+\sqrt{2}}{2}\right)$ and Diag$(T_{G''}, N)$. We made this computation for several values $N>0$ never founding any $G^{''} \prec G$ such that $G^{''}$ is also $(0,N]$-eigenvalue free. Despite the efficiency of our method, it still prohibitive, for $N=3.5$ we must to perform about 22 million tests.  In other words, we conjecture that the initial graph is \emph{minimal} in the sense  that any \emph{smaller} threshold graph is not $(0,N]$-eigenvalue free.
\begin{conj}
\label{thm:minimal}
   Let $N>0$ and $r \geq 2$ be fixed numbers defining the interval $I=(0,N]$, and $T_{G}(a_{1},\ldots,a_{r})$, the initial threshold graph  obtained from RFI$(N,r)$, then $G$  is minimal, i.e., if $\tilde{G} \preceq G$ is such that $\tilde{G}$ is also $I$-eigenvalue free then $\tilde{G} = G$.
\end{conj}
The conjecture is obviously true for  $N=\frac{-1+\sqrt{2}}{2}$ because the initial threshold graph, $T_{G}(2,1,1, \ldots,1)$, is the small one. In a future work we expect to investigate this conjecture trying to prove the minimality of the initial graph for each fixed $r$. We believe that probably does not exist a global minimum, except if we fix $r$, because we already know that the sequence $$\theta^+(T_{G}(1,2)), \theta^+(T_{G}(1,1,2)), \theta^+(T_{G}(1,1,1,2)), ...$$ converges to $N=\frac{-1+\sqrt{2}}{2}$.

\bibliographystyle{acm}
\bibliography{mybibliography}

\end{document}